\documentclass[11pt]{amsart}
\usepackage{amsmath, amssymb, amsthm, latexsym}
\usepackage{amsthm}
\usepackage{fullpage}
\usepackage{amssymb}
\usepackage{latexsym}
\usepackage[center]{caption}
\usepackage{tikz}
\usepackage{mathtools}
\usepackage{enumerate}
\usepackage{hyperref}
\usepackage{enumitem}

\newcounter{braid}
\newcounter{strands}
\pgfkeyssetvalue{/tikz/braid height}{1cm}
\pgfkeyssetvalue{/tikz/braid width}{1cm}
\pgfkeyssetvalue{/tikz/braid start}{(0,0)}
\pgfkeyssetvalue{/tikz/braid colour}{black}
\pgfkeys{/tikz/strands/.code={\setcounter{strands}{#1}}}

\makeatletter
\def\cross{%
  \@ifnextchar^{\message{Got sup}\cross@sup}{\cross@sub}}

\def\cross@sup^#1_#2{\render@cross{#2}{#1}}

\def\cross@sub_#1{\@ifnextchar^{\cross@@sub{#1}}{\render@cross{#1}{1}}}

\def\cross@@sub#1^#2{\render@cross{#1}{#2}}

\def\render@cross#1#2{
  \def\strand{#1}
  \def\crossing{#2}
  \pgfmathsetmacro{\cross@y}{-\value{braid}*\braid@h}
  \pgfmathtruncatemacro{\nextstrand}{#1+1}
  \foreach \thread in {1,...,\value{strands}}
  {
    \pgfmathsetmacro{\strand@x}{\thread * \braid@w}
    \ifnum\thread=\strand
    \pgfmathsetmacro{\over@x}{\strand * \braid@w + .5*(1 - \crossing) * \braid@w}
    \pgfmathsetmacro{\under@x}{\strand * \braid@w + .5*(1 + \crossing) * \braid@w}
    \draw[braid] \pgfkeysvalueof{/tikz/braid start} +(\under@x pt,\cross@y pt) to[out=-90,in=90] +(\over@x pt,\cross@y pt -\braid@h);
    \draw[braid] \pgfkeysvalueof{/tikz/braid start} +(\over@x pt,\cross@y pt) to[out=-90,in=90] +(\under@x pt,\cross@y pt -\braid@h);
    \else
    \ifnum\thread=\nextstrand
    \else
     \draw[braid] \pgfkeysvalueof{/tikz/braid start} ++(\strand@x pt,\cross@y pt) -- ++(0,-\braid@h);
    \fi
   \fi
  }
  \stepcounter{braid}
}

\tikzset{braid/.style={double=\pgfkeysvalueof{/tikz/braid colour},double distance=1pt,line width=2pt,white}}

\newcommand{\braid}[2][]{%
  \begingroup
  \pgfkeys{/tikz/strands=2}
  \tikzset{#1}
  \pgfkeysgetvalue{/tikz/braid width}{\braid@w}
  \pgfkeysgetvalue{/tikz/braid height}{\braid@h}
  \setcounter{braid}{0}
  \let\sigma=\cross
  #2
  \endgroup
}
\makeatother

\input xypic
\newtheorem{theorem}{Theorem}[section]
\newtheorem{proposition}[theorem]{Proposition}

\newtheorem{lemma}[theorem]{Lemma}
\newtheorem{conjecture}[theorem]{Conjecture}
\newtheorem{corollary}[theorem]{Corollary}
\theoremstyle{definition}

\newtheorem{definition}[theorem]{Definition}

\newtheorem{remark}[theorem]{Remark}
\newtheorem{example}[theorem]{Example}

\def\Z{\mathbb{Z}}

\def\Pi{\mathbb{P}^{\infty}}

\def\Zpk{\mathbb{Z}/p^{k}}
\def\Zpk1{\mathbb{Z}/p^{k-1}}

\def\sl2{\widetilde{SL_{2}(\Z)}}

\DeclareMathOperator{\ord}{ord}

\DeclareMathOperator{\Nm}{Nm}

\DeclareMathOperator{\sign}{sign}

\DeclareMathOperator{\Gal}{Gal}

\DeclareMathOperator{\Frob}{Frob}
\DeclareMathOperator{\splt}{split}
\DeclareMathOperator{\nonsplit}{nonsplit}

\DeclareMathOperator{\Aut}{Aut}

\DeclareMathOperator{\add}{add}

\DeclareMathOperator{\Sel}{Sel}
\DeclareMathOperator{\GL}{GL}
\DeclareFontFamily{U}{wncy}{}
\DeclareFontShape{U}{wncy}{m}{n}{<->wncyr10}{}
\DeclareSymbolFont{mcy}{U}{wncy}{m}{n}
\DeclareMathSymbol{\Sh}{\mathord}{mcy}{"58}

\usepackage[T2A,T1]{fontenc}
\DeclareSymbolFont{cyrillic}{T2A}{cmr}{m}{n}
\DeclareMathSymbol{\Sha}{\mathalpha}{cyrillic}{216}

\usepackage{enumitem}

\setlist[enumerate]{leftmargin=*}

\usepackage{array}

\title{Heegner points at Eisenstein primes and twists of elliptic curves}
\begin{document}
\author[Daniel Kriz]{Daniel Kriz}\email{dkriz@princeton.edu}
\address{Department of Mathematics, Princeton University, Fine Hall, Washington Rd, Princeton, NJ 08544}
\author[Chao Li]{Chao Li}\email{chaoli@math.columbia.edu} 
\address{Department of Mathematics, Columbia University, 2990 Broadway,
 New York, NY 10027}


\date{\today}

\maketitle

\begin{abstract}
Given an elliptic curve $E$ over $\mathbb{Q}$, a celebrated conjecture of Goldfeld asserts that a positive proportion of its quadratic twists should have analytic rank 0 (resp. 1). We show this conjecture holds whenever $E$ has a rational 3-isogeny. We also prove the analogous result for the sextic twists of $j$-invariant 0 curves (Mordell curves). To prove these results, we establish a general criterion for the non-triviality of the $p$-adic logarithm of Heegner points at an Eisenstein prime $p$, in terms of the relative $p$-class numbers of certain number fields and then apply this criterion to the special case $p=3$. As a by-product, we also prove the 3-part of the Birch and Swinnerton-Dyer conjecture for many elliptic curves of $j$-invariant 0.
\end{abstract}

% \tableofcontents

\section{Introduction}

\subsection{Goldfeld's conjecture}

Let $E$ be an elliptic curve over $\mathbb{Q}$. We denote by $r_\mathrm{an}(E)$ its analytic rank. By the theorem of Gross--Zagier and Kolyvagin, the rank part of the Birch and Swinnerton-Dyer conjecture holds whenever $r_\mathrm{an}(E)\in\{0,1\}$. One can ask the following natural question: how is $r_\mathrm{an}(E)$ distributed when $E$ varies in families? The simplest (1-parameter) family is given by the quadratic twists family of a given curve $E$. For a fundamental discriminant $d$, we denote by $E^{(d)}$ the quadratic twist of $E$ by $\mathbb{Q}(\sqrt{d})$. The celebrated conjecture of Goldfeld \cite{Goldfeld1979} asserts that $r_\mathrm{an}(E^{(d)})$ tends to be as low as possible (compatible with the sign of the function equation). Namely in the quadratic twists family $\{E^{(d)}\}$, $r_\mathrm{an}$ should be 0 (resp. 1) for $50\%$ of $d$'s. Although $r_\mathrm{an}\ge2$ occurs infinitely often, its occurrence should be sparse and accounts for only $0\%$ of $d$'s. More precisely, 

\begin{conjecture}[Goldfeld]\label{conj:fullgoldfeld} Let $$N_r(E, X)=\{|d|<X: r_\mathrm{an}(E^{(d)})=r\}.$$ Then 
  for $r\in\{0,1\}$, $$N_r(E,X)\sim \frac{1}{2} \sum_{|d|<X}1,\quad X\rightarrow \infty.$$ Here $d$ runs over all fundamental discriminants.
\end{conjecture}

Goldfeld's conjecture is widely open: we do not yet know a single example $E$ for which Conjecture \ref{conj:fullgoldfeld} is valid. One can instead consider the following weak version (replacing 50\% by any positive proportion):

\begin{conjecture}[Weak Goldfeld]\label{conj:weakgoldfeld}
  For $r\in \{0,1\}$, $N_r(E,X)\gg X$.
\end{conjecture}

\begin{remark}
  Heath-Brown (\cite[Thm. 4]{Heath-Brown2004}) proved Conjecture \ref{conj:weakgoldfeld}  \emph{conditional} on GRH. Recently, Smith \cite{Smith2017} has announced a proof (\emph{conditional} on BSD) of Conjecture \ref{conj:fullgoldfeld} for curves with full rational 2-torsion by vastly generalizing the works of Heath-Brown \cite{Heath-Brown1994} and Kane \cite{Kane2013}.
\end{remark}

\begin{remark}
  Katz--Sarnak \cite{Katz1999} conjectured the analogue of Conjecture \ref{conj:fullgoldfeld} for the 2-parameter family $\{E_{A,B}: y^2=x^3+Ax+B\}$ of all elliptic curves over $\mathbb{Q}$. The weak version in this case is now known \emph{unconditionally} due to the recent work of Bhargava--Skinner--W. Zhang \cite{Bhargava2014}. However, their method does not directly apply to quadratic twists families.
\end{remark}

The curve $E=X_0(19)$ is the first known example for which Conjecture \ref{conj:weakgoldfeld} is valid (see James \cite{James1998} for $r=0$ and Vatsal \cite{Vatsal1998} for $r=1$). Later many authors have verified Conjecture \ref{conj:weakgoldfeld} for infinitely many curves $E$ (see \cite{Vatsal1999}, \cite{Byeon2009} and \cite{Kriz2016}) using various methods. However, all these examples are a bit special, as they are all covered by our first main result:

\begin{theorem}[Theorem \ref{thm:av3iso}]\label{thm:3isogeny}
    The weak Goldfeld Conjecture is true for any $E$ with a rational 3-isogeny. 
\end{theorem}

% \begin{theorem}\label{thm:3isogeny}
%   The weak Goldfeld Conjecture is true for any abelian variety $A/\mathbb{Q}$ of $\GL_2$-type with a rational 3-isogeny.
% \end{theorem}

In fact, in Theorem \ref{thm:av3iso} we prove the same result for any abelian variety $A/\mathbb{Q}$ of $\GL_2$-type with a rational 3-isogeny.

\begin{remark}
  Theorem \ref{thm:3isogeny} gives so far the most general results for Conjecture \ref{conj:weakgoldfeld}. There is only one known example for which Conjecture \ref{conj:weakgoldfeld} is valid and is not covered by Theorem \ref{thm:3isogeny}: the congruent number curve $E: y^2=x^3-x$ (due to the recent work of Smith \cite{Smith2016} and Tian--Yuan--S. Zhang \cite{Tian2014a}).
\end{remark}

\begin{remark}
For explicit lower bounds for the proportion in Theorems \ref{thm:3isogeny}, see the more precise statements in Theorems \ref{twistpositiveproportion}, \ref{realtwistpositiveproportion}, Proposition \ref{semistableprop}, and Example \ref{exa:19a1}.
\end{remark}

For an elliptic curve $E$ of $j$-invariant 0 (resp. 1728), one can also consider its cubic or sextic (resp. quartic) twists family. The weak Goldfeld conjecture in these cases asserts that for $r\in \{0,1\}$, a positive proportion of (higher) twists should have analytic rank $r$. Our second main result verifies the weak Goldfeld conjecture for the sextic twists family. More precisely, consider the elliptic curve $$E=X_0(27): y^2=x^3-432$$ of $j$-invariant 0 (isomorphic to the Fermat cubic $X^3+Y^3=1$). For a 6th-power-free integer $d$, we denote by $$E_d: y^2=x^3-432d$$ the $d$-th sextic twist of $E$. These $E_d$'s are also known as Mordell curves.

\begin{theorem}[Corollary \ref{thm:sexticdensity}]\label{thm:sexticgoldfeldmain}
  The weak Goldfeld conjecture is true for the sextic twists family $\{E_d\}$.   In fact, $E_d$ has analytic rank $0$ (resp. $1$) for at least 1/6 of fundamental discriminants $d$.% : for $r\in\{0,1\}$, there exists a positive proportion of 6th-power-free integers $d$ such that $r_\mathrm{an}(E_d)=r$.
\end{theorem}

\begin{remark}
For a wide class of elliptic curves of $j$-invariant 0, we can also construct many (in fact $\gg X/\log^{7/8}X$) cubic twists of analytic rank 0 (resp. 1). However, these cubic twists do not have positive density. See the more precise statement in Theorem \ref{thm:cubictwists} and Example \ref{exa:cubictwist108}.
\end{remark}

\begin{remark}
  In a recent work, Bhargava--Elkies--Shnidman \cite{Bhargava2016} prove the analogue of Theorem \ref{thm:sexticgoldfeldmain} for \emph{3-Selmer ranks} 0,1, by determining the exact average size of 3-isogeny Selmer groups (its boundness was first proved by Fouvry \cite{Fouvry1993}). The same method also works for quadratic twists family of elliptic curves and $\GL_2$-type abelian varieties with a 3-isogeny (\cite{Bhargava2017}, \cite{Shnidman2017}). We remark that their method however does not have the same implication for analytic rank $r=0,1$ (or algebraic rank 1), since the $p$-converse to the theorem of Gross--Zagier and Kolyvagin is not known for $p$ an additive and Eisenstein prime.
\end{remark}

\begin{remark}Recently, Browning \cite{Browning2017} has used Theorem \ref{thm:sexticgoldfeldmain} as key input in his argument to show that a positive proportion (when ordered by height) of smooth projective cubic surfaces of the form $f(x_0,x_1) = g(x_2,x_3)$, where $f, g$ are binary cubic forms over $\mathbb{Q}$, have a $\mathbb{Q}$-rational point. % This result drastically increases the set of known cases of cubic surfaces which have a $\mathbb{Q}$-rational point, and gives a very uniform family of such examples.
\end{remark}

\subsection{Heegner points at Eisenstein primes} The above results on weak Goldfeld conjecture are applications of a more general $p$-adic criterion for non-triviality of Heegner points on $E$ (applied to $p=3$). To be more precise, let $E/\mathbb{Q}$ be an elliptic curve of conductor $N$. Let $K=\mathbb{Q}(\sqrt{d_K})$ denote an imaginary quadratic field of fundamental discriminant $d_K$. We assume that $K$ satisfies the \emph{Heegner hypothesis for $N$}:
\begin{center}
each prime factor $\ell$ of $N$ is split in $K$.
\end{center}
For simplicity, we also assume that $d_K\ne-3,-4$ so that $\mathcal{O}_K^{\times} = \{\pm 1\}$, and that $d_K$ is odd (i.e. $d_K \equiv 1 \pmod{4}$). We denote by $P\in E(K)$ the corresponding Heegner point, defined up to sign and torsion with respect to a fixed modular parametrization $\pi_E: X_0(N)\rightarrow E$ (see \cite{Gross1984}). Let $$f(q)=\sum_{n=1}^\infty a_n(E) q^n\in S_2^\mathrm{new}(\Gamma_0(N))$$ be the normalized newform associated  to $E$. Let $\omega_E\in \Omega_{E/\mathbb{Q}}^1 := H^0(E/\mathbb{Q},\Omega^1)$ such that $$\pi_E^*(\omega_E)= f(q) \cdot dq/q.$$ We denote by $\log_{\omega_E}$ the formal logarithm associated to $\omega_E$. Notice $\omega_E$ may differ from the N\'{e}ron differential by a scalar when $E$ is not the optimal curve in its isogeny class. 

For a finite order Galois character $\psi: G_\mathbb{Q}:=\Gal(\overline{\mathbb{Q}}/\mathbb{Q})\rightarrow\overline{\mathbb{Q}}^\times$, we abuse notation and denote by $\psi: (\mathbb{Z}/f \mathbb{Z})^\times\rightarrow \mathbb{C}^\times$ the corresponding Dirichlet character, where $f$ is its conductor. The generalized (first) Bernoulli number is defined to be
\begin{equation}
  \label{Bernoulliformula}
  B_{1,\psi}:=\frac{1}{f}\sum_{m=1}^{f}\psi(m)m.
\end{equation}
 Let $\varepsilon_K$ be the quadratic character associated to $K$. We consider the even Dirichlet character \begin{align*}
\psi_0 := \begin{cases}\psi, & \text{if } \psi \text{ is even},\\
\psi\varepsilon_K, & \text{if }  \psi \text{ is odd}.\\
\end{cases}
\end{align*} 

%  $(\mathbb{Z}/f)^\times\rightarrow \mathbb{C}^\times$, where $f=f(\psi)$ is the conductor of $\psi$. As usual we extend $\psi$ to $\mathbb{Z}/f$ by setting $\psi(a)=0$ if $\gcd(a,f)\ne1$

%  For Dirichlet characters $\psi_1$ and $\psi_2$, we let $\psi_1\psi_2$ denote the unique primitive Dirichlet character $\psi$ such that $\psi(a) = \psi_1(a)\psi_2(a)$ for all $a \in \mathbb{Z}$ with $(a,f(\psi)) = 1$.

Now suppose $p$ is an \emph{Eisenstein prime} for $E$ (i.e., $E[p]$ is a reducible $G_\mathbb{Q}$-representation, or equivalently, $E$ admits a rational $p$-isogeny), we prove the following criterion for the non-triviality of the $p$-adic logarithm of Heegner points, in terms of the $p$-indivisibility of Bernoulli numbers.

\begin{theorem}[Theorem \ref{thm:Heegnercorollary}]\label{thm:mainthm}
Let $E/\mathbb{Q}$ be an elliptic curve of conductor $N$.  Suppose $p$ is an odd prime such that $E[p]$ is a reducible $G_\mathbb{Q}$-representation. Write $E[p]^\mathrm{ss}\cong \mathbb{F}_p(\psi) \oplus \mathbb{F}_p(\psi^{-1}\omega),$ for some character $\psi: G_\mathbb{Q}\rightarrow \Aut(\mathbb{F}_p)\cong\mu_{p-1}$ and the mod $p$ cyclotomic character $\omega$. Assume that
  \begin{enumerate}
  \item $\psi(p)\ne1$ and $(\psi^{-1}\omega)(p)\ne1$.
  \item $E$ has no primes of split multiplicative reduction.
  \item If $\ell\ne p$ is an additive prime for $E$, then $\psi(\ell)\ne1$ and $(\psi^{-1}\omega)(\ell)\ne1$.
  \end{enumerate}
  Let $K$ be an imaginary quadratic field satisfying the Heegner hypothesis for $N$. Let $P\in E(K)$ be the associated Heegner point. Assume $p$ splits in $K$. Assume $$B_{1, \psi_0^{-1}\varepsilon_K}\cdot B_{1, \psi_0\omega^{-1}}\ne0\pmod{p}.$$ Then $$\frac{|\tilde E^\mathrm{ns}(\mathbb{F}_p)|}{p}\cdot \log_{\omega_E}P\ne0\pmod{p}.$$ In particular,  $P\in E(K)$ is of infinite order and $E/K$ has analytic and algebraic rank 1.
\end{theorem}

In fact, this is a specialization of the most general form of our main result given in Theorem \ref{thm:Heegnercorollary}, which addresses abelian varieties of $\GL_2$-type over $\mathbb{Q}$.

\begin{remark} When $E/\mathbb{Q}$ has CM by $\mathbb{Q}(\sqrt{-p})$ (of class number 1), Rubin \cite{Rubin1983} proved a mod $p$ congruence formula between the algebraic part of $L(E,1)$ and certain Bernoulli numbers. Notice that $E$ admits a $p$-isogeny (multiplication by $\sqrt{-p}$), Theorem \ref{thm:mainthm} specializes to provide a mod $p$ congruence between the $p$-adic logarithm of the Heegner point on $E$ and certain Bernoulli numbers, which can be viewed as a generalization of Rubin's formula from the rank 0 case to the \emph{rank 1} case.
\end{remark}

Notice that the two odd Dirichlet characters $\psi_0^{-1}\varepsilon_K$ and $\psi_0\omega^{-1}$ cut out two abelian CM fields (of degree dividing $p-1$). When the relative $p$-class numbers of these two CM fields are trivial, it follows from the relative class number formula that the two Bernoulli numbers in Theorem \ref{thm:mainthm} are nonzero mod $p$ (see \S \ref{sec:bern-numb-relat}), hence we conclude $r_\mathrm{an}(E/K)=1$. When $p=3$, the relative $p$-class numbers becomes the 3-class numbers of two quadratic fields. Our final ingredient to finish the proof of Theorems \ref{thm:3isogeny} and \ref{thm:mainthm} is Davenport--Heilbronn's theorem (\cite{Davenport1971}) (enhanced by Nakagawa--Horie \cite{Nakagawa1988} with congruence conditions), which allows one to find a positive proportion of twists such that both 3-class numbers in question are trivial.

\subsection{A by-product: the 3-part of the BSD conjecture} The Birch and Swinnerton-Dyer conjecture predicts the precise formula \begin{equation}
  \label{eq:bsdformula}
\frac{L^{(r)}(E/\mathbb{Q},1)}{r!\Omega(E/\mathbb{Q}) R(E/\mathbb{Q})}=\frac{\prod_p c_p(E/\mathbb{Q})\cdot |\Sha(E/\mathbb{Q})|}{|E(\mathbb{Q})_\mathrm{tor}|^2}  
\end{equation} for the leading coefficient of the Taylor expansion of $L(E/\mathbb{Q},s)$ at $s=1$ (here $r=r_\mathrm{an}(E)$) in terms of various important arithmetic invariants of $E$ (see \cite{Gross2011} for detailed definitions). When $r\le1$, both sides of the BSD formula (\ref{eq:bsdformula}) are known to be positive rational numbers. To prove that (\ref{eq:bsdformula}) is indeed an equality, it suffices to prove that it is an equality up to a $p$-adic unit, for each prime $p$. This is known as the \emph{$p$-part of the BSD formula} (BSD($p$) for short). Much progress has been made recently, but only in the case $p$ is \emph{semi-stable} and \emph{non-Eisenstein}% (for $r=0$: \cite{Kato2004},  \cite{Skinner2014}, \cite{Wan2014a}, \cite{Sprung2016}; for $r=1$: \cite{Zhang2014}, \cite{Skinner2014a},\cite{Berti2016}, \cite{Jetchev2015}, \cite{Sprung2016}, \cite{Castella2017})
. We establish the following new results on BSD(3) for many sextic twists $E_d: y^2=x^3-432d$,  in the case $p=3$ is \emph{additive} and \emph{Eisenstein}.

\begin{theorem}[Theorem \ref{thm:BSD3}]\label{thm:bsdmain}
  Suppose $K$ is an imaginary quadratic field satisfies the Heegner hypothesis for $3d$. Assume that
  \begin{enumerate}
  \item $d$ is a fundamental discriminant.
  \item $d\equiv2,3,5,8\pmod{9}$.
  \item If $d>0$, $h_3(-3d)=h_3(d_Kd)=1$. If $d<0$, $h_3(d)=h_3(-3d_Kd)=1$. 
  \item The Manin constant of $E_d$ is coprime to 3.
  \end{enumerate}
Then $r_\mathrm{an}(E_d/K)=1$ and BSD(3) holds for $E_d/K$. (Here $h_3(D)$ denotes the 3-class number of $\mathbb{Q}(\sqrt{D})$.)
\end{theorem}

\begin{remark}
Since the curve $E_d$ has complex multiplication by $\mathbb{Q}(\sqrt{-3})$, we already know that BSD($p$) holds for $E_d/\mathbb{Q}$ if $p\ne2,3$ (when $r=0$) and if $p\ne2,3$ is a prime of good reduction or potentially good ordinary reduction (when $r=1$) thanks to the works \cite{Rubin1991}, \cite{Perrin-Riou1987}, \cite{Kobayashi2013}, \cite{Pollack2004}, \cite{Li2016}. When $r=0$, we also know BSD(3) for some quadratic twists of the two curves $X_0(27)$ and $X_0(36)$ of $j$-invariant 0, using explicit weight 3/2 modular forms (\cite{Nekovavr1990}, \cite{Ono1998a}, \cite{James1999}).
\end{remark}

\subsection{Comparison with previous methods establishing the weak Goldfeld conjecture}
  \begin{enumerate}
  \item The work of James \cite{James1998} on weak Goldfeld for $r=0$ uses Waldspurger's formula relating coefficients of weight $3/2$ modular forms and quadratic twists $L$-values (see also Nekov{\'a}{\v{r}} \cite{Nekovavr1990}, Ono--Skinner \cite{Ono1998}). Our proof does not use any half-integral weight modular forms. 
  \item   When $N$ is a prime different from $p$, Mazur in his seminal paper \cite{Mazur1979} proved a congruence formula at an Eisenstein prime above $p$, between the algebraic part of $L(J_0(N),\chi,1)$ and a quantity involving generalized Bernoulli numbers attached to $\chi$, for certain odd Dirichlet characters $\chi$. This was later generalized by Vatsal \cite{Vatsal1999} for more general $N$ and used to prove weak Goldfeld for $r=0$ for infinitely many elliptic curves.
  \item When $N$ is a prime different from $p$, Mazur \cite{Mazur1979}  also constructed a point of infinite order on the Eisenstein quotient of $J_0(N)$, when certain quadratic class number is not divisible by $p$. This was later generalized by Gross \cite[II]{Gross1984} to more general $N$, and became the starting point of the work of Vatsal \cite{Vatsal1998} and Byeon--Jeon--Kim \cite{Byeon2009} on weak Goldfeld for $r=1$.  
  \item Our main congruence at Eisenstein primes (see \S \ref{maincongruence}) through which Theorem \ref{thm:mainthm} is established can be viewed as a vast generalization of Mazur's congruence from $J_0(N)$ to \emph{any} elliptic curve with a $p$-isogeny and to \emph{both} rank 0 and rank 1 case. To achieve this, instead of working with $L$-functions directly, we use the $p$-adic logarithm of Heegner points as the $p$-adic incarnation of $L$-values (or $L$-derivatives).% , as considered by Bertolini--Darmon--Prasanna \cite{Bertolini2013} (see also \cite{Liu2014}).
  \item The recent work \cite{Kriz2016} also uses $p$-adic logarithm of Heegner points. As we have pointed out, the crucial difference is that our proof uses a direct method of $p$-adic integration, and does not rely on the deep $p$-adic Gross--Zagier formula of  \cite{Bertolini2013}. This is the key observation to remove \emph{all} technical hypothesis appeared in previous works, which in particular makes the application to the sextic twists family possible.% , which was inaccessible before.
  \item Although the methods are completely different, the final appearance of Davenport--Heilbronn type theorem is a common feature in all previous works (\cite{James1998}, \cite{Vatsal1998}, \cite{Vatsal1999}, \cite{Byeon2009}, \cite{Kriz2016}), and also ours.  
  \end{enumerate}

 \subsection{Strategy of the proof}
The proof of Theorem \ref{thm:mainthm} (and the more general version Theorem \ref{thm:Heegnercorollary}) relies on the main congruence identity (\S\ref{maincongruence}) between the $p$-adic logarithm of Heegner points and a product of two Bernoulli numbers.

The starting point is that the prime $p$ being Eisenstein produces a congruence between the modular form $f$ and a weight 2 Eisenstein series $g$, away from the bad primes. We then apply certain Hecke operators (which we call \emph{stabilization operators}) in order to produce a modified Eisenstein series $g^{(N)}$ whose entire $q$-expansion $g^{(N)}(q)$ is congruent to $f(q)$.
  Applying another $p$-stabilization operator and the Atkin-Serre derivatives $\theta^j$, we obtain a $p$-adically continuously varying system of congruences $\theta^jf^{(p)}(q) \equiv \theta^jg^{(pN)}(q) \pmod{p}$.   By the $q$-expansion principle and our assumption that $p$ splits in $K$, we can sum this congruence over CM points to obtain a congruence between a normalized CM period sum and a $p$-adic Katz $L$-value times certain Euler factors at bad primes.

Taking $j \rightarrow -1$ ($p$-adically), the CM period sums converge to the $p$-adic logarithm of the Heegner point times an Euler factor at $p$, by Coleman's integration.  The Katz $L$-values converge to a product of two Bernoulli numbers, by Gross's factorization. We finally arrive at the main congruence identity.

\subsection{Structure of the paper} In \S \ref{sec:heegner-points-at}, we establish the non-triviality criterion for Heegner points at Eisenstein primes, in terms of $p$-indivisibility of Bernoulli numbers (Theorem \ref{thm:mainthm}). In \S \ref{sec:bern-numb-relat}, we explain the relation between the Bernoulli numbers and relative class numbers. In \S \ref{sec:goldf-conj-ellipt}, we combine our criterion and the Nakagawa--Horie theorem to prove the weak Goldfeld conjecture for abelian varieties of $\GL_2$-type with a 3-isogeny (Theorem \ref{thm:av3iso}). In \S \ref{sec:sextic-twists-family}, we give applications to the sextic twists family (Theorems \ref{thm:sexticgoldfeldmain} and \ref{thm:bsdmain}). Finally, in \S \ref{sec:cubic-twists-famil}, we give an application to cubic twists families (Theorem \ref{thm:cubictwists}).

\subsection{Acknowledgments}  We are grateful to M. Bhargava, D. Goldfeld, B. Gross, B. Mazur, K. Prasanna, P. Sarnak, A. Shnidman, C. Skinner,  E. Urban, X. Wan, A. Wiles and S. Zhang for helpful conversations or comments. The examples in this article are computed using Sage (\cite{sage}).

\section{Heegner points at Eisenstein primes}
\label{sec:heegner-points-at}In this section, we carry out the $p$-adic integration which makes up the heart of Theorem \ref{thm:mainthm}. In the course of our argument, we recall certain Hecke operators from \cite[Section 2]{KrizLi2016a} which we refer to as ``stabilization operators''. These operators will be used to modify $q$-expansions at bad primes to translate an isomorphism of mod $p$ Galois representations to a system of congruences of $p$-adic modular forms. We begin by recalling some notation which will be used throughout this section.

%We will let $K=\mathbb{Q}(\sqrt{d_K})$ denote an imaginary quadratic field of fundamental discriminant $d_K$ satisfying the \emph{Heegner hypothesis for $N$}:
%\begin{center}
%each prime factor $\ell$ of $N$ is split in $K$.
%\end{center}
%For simplicity, we will also assume that $d_K < -4$ so that $\mathcal{O}_K^{\times} = \{\pm 1\}$ and that $d_K$ is odd (i.e. $d_K \equiv 1 \mod 4$).

\subsection{Notations and conventions} Fix an algebraic closure $\overline{\mathbb{Q}}$ of $\mathbb{Q}$, and view all number fields $L$ as embedded $L \subset \overline{\mathbb{Q}}$. Let $h_L$ denote the class number of $L$, and let $\overline{\mathbb{Z}}$ denote the integral closure of $\mathbb{Z}$ in $\overline{\mathbb{Q}}$. Fix an algebraic closure $\overline{\mathbb{Q}}_p$ of $\mathbb{Q}_p$ (which amounts to fixing a prime of $\overline{\mathbb{Q}}$ above $p$). Let $\mathbb{C}_p$ be the $p$-adic completion of $\overline{\mathbb{Q}}_p$, and let $L_p$ denote the $p$-adic completion of $L \subset \mathbb{C}_p$. For any integers $a, b$, let $(a,b)$ denote their (positive) greatest common divisor. Given ideals $\mathfrak{a}, \mathfrak{b} \subset \mathcal{O}_L$, let $(\mathfrak{a},\mathfrak{b})$ denote their greatest common divisor.

All Dirichlet (i.e. finite order) characters $\psi : \mathbb{A}_{\mathbb{Q}}^{\times} \rightarrow \overline{\mathbb{Q}}^{\times}$ will be primitive, and we denote the conductor by $f(\psi)$, which as an ideal in $\mathbb{Z}$ identified with its unique positive generator. We may equivalently view $\psi$ as a character $\psi : (\mathbb{Z}/f(\psi))^{\times} \rightarrow \overline{\mathbb{Q}}^{\times}$ via
$$\psi(x \mod f(\psi)) = \prod_{\ell\nmid f(\psi)} \psi_{\ell}(x) = \prod_{\ell|f(\psi)}\psi_{\ell}^{-1}(x)$$
where $\psi_{\ell} : \mathbb{Q}_{\ell}^{\times} \rightarrow \overline{\mathbb{Q}}^{\times}$ is the local character at $\ell$. Following convention, we extend $\psi$ to $\mathbb{Z}/f(\psi) \rightarrow \overline{\mathbb{Q}}$, defining $\psi(a) = 0$ if $(a,f(\psi)) \neq 1$. Given Dirichlet character $\psi_1$ and $\psi_2$, we let $\psi_1\psi_2$ denote the unique primitive Dirichlet character such that $\psi_1\psi_2(a) = \psi_1(a)\psi_2(a)$ for all $a \in \mathbb{Z}$ with $(a,f(\psi)) = 1$. Given a prime $p$, let $f(\psi)_p$ denotes the $p$-primary part of $f(\psi)$ and let $f(\psi)^{(p)}$ denote the prime-to-$p$ part of $f(\psi)$.

We define the Gauss sum $\mathfrak{g}(\psi)$ of $\psi$ and local Gauss sums $\mathfrak{g}_{\ell}(\psi)$ as in \cite[Section 1]{Kriz2016}. We will often identify a Dirichlet character $\psi : \mathbb{A}_{\mathbb{Q}}^{\times} \rightarrow \overline{\mathbb{Q}}^{\times}$ with its associated Galois character $\psi : \text{Gal}(\overline{\mathbb{Q}}/\mathbb{Q}) \rightarrow \overline{\mathbb{Q}}^{\times}$ via the (inverse of the) Artin reciprocity map $\text{Gal}(\overline{\mathbb{Q}}/\mathbb{Q}) \rightarrow \text{Gal}(\overline{\mathbb{Q}}/\mathbb{Q})^{\text{ab}} \xrightarrow{\sim} \hat{\mathbb{Z}}^{\times}$, using the arithmetic normalization (i.e. the normalization where $\Frob_{\ell}$, the Frobenius conjugacy class at $\ell$, gets sent to the id\'{e}le which is $\ell$ at the place of $\mathbb{Z}$ corresponding to $\ell$ and 1 at all other places).
%For an imaginary quadratic field $L$, let $\varepsilon_L$ denote the quadratic character associated with $L$. For quadratic fields $L$ and $L'$ with fundamental discriminants $d$ and $d_{L'}$, we let $L\cdot L' = \mathbb{Q}(\sqrt{dd_{L'}})$; hence under our conventions, $\varepsilon_{L\cdot L'} = \varepsilon_L\varepsilon_{L'}$. 
Throughout, for a given $p$, let $\omega : \mathrm{Gal}(\overline{\mathbb{Q}}/\mathbb{Q}) \rightarrow \mu_{p-1}$ denote the mod $p$ cyclotomic character. %  i.e. for $x \in (\mathbb{Z}/p)^{\times}$, $\omega(x) \in \mu_{p-1}$ such that $\omega(x) \equiv x \pmod{p}$.
Let $\mathbb{N}_{\mathbb{Q}} : \mathbb{A}_{\mathbb{Q}}^{\times} \rightarrow \mathbb{C}^{\times}$ denote the norm character, normalized to have infinity type $-1$. For a number field $K$, let $\Nm_{K/\mathbb{Q}} : \mathbb{A}_K^{\times} \rightarrow \mathbb{A}_{\mathbb{Q}}^{\times}$ denote the id\`{e}lic norm, and let $\mathbb{N}_K := \mathbb{N}_{\mathbb{Q}} \circ \Nm_{K/Q} : \mathbb{A}_K^{\times} \rightarrow \mathbb{C}^{\times}$. Suppose we are given an imaginary quadratic field $K$ with fundamental discriminant $d_K$. Let $\varepsilon_K : (\mathbb{Z}/d_K)^{\times} \rightarrow \mu_2$ be the quadratic character associated with $K$. For any Dirichlet character $\psi$ over $\mathbb{Q}$, let 
\begin{align*}
\psi_0 := \begin{cases}\psi, & \text{if} \; \psi \; \text{even},\\
\psi\varepsilon_K, & \text{if} \; \psi \; \text{odd}.\\
\end{cases}
\end{align*}

Throughout, let $E/\mathbb{Q}$ be an elliptic curve of conductor $N = N_{\mathrm{split}}N_{\mathrm{nonsplit}}N_{\mathrm{add}}$, where $N_{\mathrm{split}}$ is only divisible by primes of split multiplicative reduction, $N_{\mathrm{nonsplit}}$ is only divisible by primes of nonsplit multiplicative reduction, and $N_{\mathrm{add}}$ is only divisible by primes of additive reduction.

Finally, for any number field $L$, let $h_L$ denote its class number. For any non-square integer $D$, we denote by $h_3(D):=|\mathrm{Cl}(\mathbb{Q}(\sqrt{D}))[3]|$ the 3-class number of the quadratic field $\mathbb{Q}(\sqrt{D})$.  

%>>>>>>> .r147
\subsection{Main theorem}\label{sec:mainthm}

We will show, by direct $p$-adic integration, the following generalization of Theorem 13 of loc. cit.\footnote{Here our generalization also corrects a self-contained typo in the statement of Theorem 13 in loc. cit., where part of condition (3) was mistranscribed from Theorem 7 in loc. cit.: ``$\ell \not\equiv -1\mod p$'' should be ``$\ell \not\equiv \psi(\ell) \mod p$''.}% We will show, by direct $p$-adic integration, the following generalization of \cite[Theorem 13]{Kriz2016}\footnote{Here our generalization also corrects a minor typo in the statement of Theorem 13 in loc. cit.}. % , which was a corollary of the the main theorem (Theorem 3 and Corollary 4) in loc. cit. 
% Our generalization of Theorem 13 in loc. cit.
Our generalization, in particular, does not require $p\nmid N$.

The most general form of our result will address $\GL_2$-type abelian varieties attached to normalized newforms of weight 2. Let $f\in S_2(\Gamma_0(N))$ be a normalized newform, with associated $q$-expansion at $\infty$ given by $\sum_{n = 1}^{\infty}a_nq^n$. Suppose $\lambda$ is the prime above $p$ in the ring of integers of the number field $E_f$ generated by the Hecke eigenvalues of $f$ which is fixed by our above choice of embeddings $E_f \subset \overline{\mathbb{Q}} \hookrightarrow \overline{\mathbb{Q}}_p$. Henceforth, let $\mathbb{F}_{\lambda}$ denote the residue field of $E_f$ at $\lambda$. Let $\rho_f$ be the semisimple $\lambda$-adic $G_{\mathbb{Q}}$-representation associated with $f$, and let $\bar{\rho}_f$ denote its mod $\lambda$ reduction. We let $A_f$ denote the $\GL_2$-type abelian variety associated with $f$ by Eichler-Shimura theory (defined uniquely up to isogeny over $\mathbb{Q}$). In the rest of the article, when we say $A$ is an \emph{abelian variety of $\GL_2$-type}, we always mean $A$ is chosen in its isogeny class so that $A$ admits an action by the ring of integers of $E_f$. Let $\pi_f : J_0(N) \rightarrow A$ be a modular parametrization. Let $\omega_f := f(q)\frac{dq}{q} \in \Omega_{X_0(N)/\mathbb{Q}}^1$, and let $\omega_{A} \in \Omega_{A/\mathbb{Q}}^1$ be such that $\pi_f^*\omega_{A} = \omega_f$.

Henceforth, write $N = N_+N_-N_0$, where
\begin{enumerate}
\item $\ell | N_+ \implies a_{\ell}\equiv \psi(\ell) \pmod{\lambda}$,
\item $\ell | N_- \implies a_{\ell} \equiv \psi^{-1}(\ell)\ell \pmod{\lambda}$,
\item $\ell | N_0 \implies a_{\ell} \equiv 0 \pmod{\lambda}$.
\end{enumerate}
When $\bar{\rho}_f$ is reducible, such a decomposition always exists, by Theorem 34 of loc. cit. When $f$ is attached to an elliptic curve $E/\mathbb{Q}$, for example, we can take $N_+ = N_{\splt}, N_- = N_{\nonsplit}$ and $N_0 = N_{\add}$, where
\begin{enumerate}
\item $\ell | N_{\splt} \implies E$ has split multiplicative reduction at $\ell$,
\item $\ell | N_- \implies E$ has nonsplit multiplicative bad reduction at $\ell$,
\item $\ell | N_0 \implies E$ has additive bad reduction at $\ell$.
\end{enumerate}

% For any abelian variety $A/\mathbb{Q}$ of $\GL_2$-type, by the modularity theorem we know that $A$ is isogenous over $\mathbb{Q}$ to some $A_f$. Let $\phi: A_f \rightarrow A$ be such an isogeny, and $\omega_A \in \Omega_{A/\mathbb{Q}}^1$ such that $\phi^*\omega_A = \omega_{A_f}$.

\begin{theorem}\label{thm:Heegnercorollary} 
Let $A/\mathbb{Q}$ be an abelian variety of $\GL_2$-type (satisfying our assumptions above). Assume that $A[\lambda]$ is reducible, or equivalently, $A[\lambda]^{\mathrm{ss}} \cong \mathbb{F}_{\lambda}(\psi)\oplus \mathbb{F}_{\lambda}(\psi^{-1}\omega)$, for some character $\psi : \mathrm{Gal}(\overline{\mathbb{Q}}/\mathbb{Q}) \rightarrow \mathbb{F}_\lambda^{\times}$. Let $K$ be an imaginary quadratic field satisfying the Heegner hypothesis for $N$. Suppose $p$ splits in $K$. Suppose further that either the following conditions hold
\begin{enumerate}
\item $\psi(p) \neq 1$ and $(\psi^{-1}\omega)(p) \neq 1$,
\item $N_+ = 1$,
\item $8\nmid f(\psi_0)$ if $p = 2$, and $p^2 \nmid f(\psi_0)$ if $p > 2$,
\item $\ell\neq p, \ell|N_0$ implies either $\psi(\ell) \neq 1$ and $\ell \not\equiv \psi(\ell) \pmod{\lambda}$, or $\psi(\ell) = 0$,
\item $p\nmid B_{1,\psi_0^{-1}\varepsilon_K}\cdot B_{1,\psi_0\omega^{-1}}$,
\end{enumerate}
or the following conditions hold
\begin{enumerate}
\item $\psi = 1$,
\item $p|N$,
\item $\ell|N, \ell \neq p$ implies $\ell||N, \ell\equiv -1 \pmod{p}, \ell \not\equiv 1 \pmod{p}$
\item $\ord_{\lambda}\left(\frac{p-1}{2p}\log_p\overline{\alpha}\right) = 0$,
\end{enumerate}
where $\alpha \in \mathcal{O}_K^{\times}$ and $(\alpha) = \mathfrak{p}^{h_K}$, $\overline{\alpha}$ is its complex conjugate, and $\log_p$ is the Iwasawa $p$-adic logarithm. 

Let $P\in A(K)$ be the associated Heegner point. Then $$\frac{1+p-a_p}{p}\cdot \log_{\omega_A}P\ne0\pmod{p\mathcal{O}_{K_p}}.$$ In particular,  $P\in A(K)$ is of infinite order and $A/K$ has analytic and algebraic rank $\dim A$.
\end{theorem}

\begin{remark}\label{remark:ellipticcurve}Suppose that $\psi \neq 1$, and $A = E$ is an elliptic curve (so that $\lambda = p$). Then one can show that condition (3) for the case $\psi \neq 1$ in the statement of Theorem \ref{thm:Heegnercorollary}, by the following argument.

If $p = 2$, then $\psi_0 = 1$ and $f(\psi_0) = 1$. If $p = 3$, then $\psi_0 : \Gal(\overline{\mathbb{Q}}/\mathbb{Q}) \rightarrow \mu_2$ is quadratic, and so $9 \nmid f(\psi_0)$ (since $f(\psi_0)$ is squarefree outside of 2). If $p \ge 5$, then since $E[p]^{\mathrm{ss}} \cong \mathbb{F}_p(\psi)\oplus \mathbb{F}_p(\psi^{-1}\omega)$, then $f(\psi)\cdot f(\psi^{-1}\omega)|N$. Since $p$ splits in $K$, $f(\varepsilon_K)_p = 1$, and so $f(\psi_0)_p = f(\psi)_p$. Since $f(\omega) = p$, we have $f(\psi^{-1}\omega)_p = f(\psi^{-1})_p = f(\psi)_p$, and hence $f(\psi)_p^2|N$. Now assume for the sake of contradiction that $p^2|f(\psi_0)$. Then since $p^2|f(\psi_0)_p = f(\psi)_p$, we have $p^4|f(\psi)_p^2|N$. However since $N$ is the conductor of $E/\mathbb{Q}$ and $p \ge 5$, we have $\ord_p(N) \le 2$, a contradiction.
\end{remark}

\begin{remark}When $p = 2$ and $A = E$ is an elliptic curve, we must have $\psi = 1$ (since $\psi : \Gal(\overline{\mathbb{Q}}/\mathbb{Q}) \rightarrow \mu_{p-1} = \{1\}$). Note also that by (3) of the second part of Theorem \ref{thm:Heegnercorollary}, in this case $N$ must be a power of 2.
%Note that when $p = 2$, the right-hand side of (\ref{nontrivcong}) is always $0 \pmod{2}$. (By consideration of the Euler factors, genus theory and Neumann-Setzer's theorem on elliptic curves of prime conductor.)
\end{remark}

\begin{remark}\label{remark:3isogeny}Suppose $p = 3$, and that the $\GL_2$-type abelian variety $A/\mathbb{Q}$ has a 3-isogeny defined over $\mathbb{Q}$ (i.e., $\mathbb{F}_\lambda\cong \mathbb{F}_3$). Then $\psi$ is necessarily quadratic as is $\psi_0$, and so $9\nmid f(\psi_0)$, and condition (3) in the case $\psi \neq 1$ of the statement of Theorem \ref{thm:Heegnercorollary} is satisfied.
\end{remark}

\begin{remark}\label{p=3remark}Note that when $p = 3$ and $\psi$ is quadratic, condition (3) in case $\psi \neq 1$ of the statement of Theorem \ref{thm:Heegnercorollary} is equivalent to 
\begin{itemize}
\item $\ell|N_{\add}, \ell \equiv 1 \pmod{3}$ implies that $\psi(\ell) = -1$, and 
\item $\ell\neq 3, \ell|N_{\add}, \ell \equiv 2 \pmod{3}$ implies that $\psi(\ell) = 0$. 
\end{itemize}
\end{remark}

\subsection{Stabilization operators}\label{sec:stabilization}Here, we recall the definition of ``stabilization operators'', as in \cite[\S 2.3]{KrizLi2016a}. We will use Katz's notion of $p$-adic modular forms as rules on the moduli space of isomorphism classes of ordinary test triples (see \cite[Definition 2.1 and 2.2]{KrizLi2016a}). Let $\tilde{M}_k^{\text{$p$-adic}}(\Gamma_0(N))$ denote the space of weak $p$-adic modular forms of level $N$ and $M_k^{\text{$p$-adic}}(\Gamma_0(N))$ the space of $p$-adic modular forms of level $N$, respectively. (See the paragraph after Definition 3.2 in loc. cit.) Note that $M_k^{\text{$p$-adic}}(\Gamma_0(N)) \subset \tilde M_k^{\text{$p$-adic}}(\Gamma_0(N))$. Fix $N^{\#} \in \mathbb{Z}_{>0}$ such that $N|N^{\#}$ (so that we may view $F \in M_k^{\text{$p$-adic}}(\Gamma_0(N^{\#}))$) and suppose that $\ell$ is a prime such that $\ell^2|N^{\#}$.  Let $V_{\ell}$ be as defined in \S 3.3 of loc. cit.
 
Now we define the stabilization operators as operations on rules on the moduli space of isomorphism classes of test triples. Let $F \in \tilde{M}_k^{\text{$p$-adic}}(\Gamma_0(N))$ and henceforth suppose $N$ is the \emph{minimal} level of $F$. View $F \in \tilde{M}_k^{\text{$p$-adic}}(\Gamma_0(N^{\#}))$, and let $a_{\ell}(F)$ denote the coefficient of the $q^{\ell}$ term in the $q$-expansion $F(q)$. Then up to permutation there is a unique pair of numbers $(\alpha_{\ell}(F), \beta_{\ell(F)}) \in \mathbb{C}_p^2$ such that $\alpha_{\ell}(F) + \beta_{\ell}(F) = a_{\ell}(F)$, $\alpha_{\ell}(F)\beta_{\ell}(F) = \ell^{k-1}$. We henceforth fix an ordered pair $(\alpha_{\ell},\beta_{\ell})$. 

\begin{definition}\label{stabilization}When $\ell\nmid N$, we define the \emph{$(\ell)^+$-stabilization of $F$} as
\begin{equation}\label{eq:stabilizationmoduli-1}
F^{(\ell)^+} = F - \beta_{\ell}(F)V_{\ell}^*F,
\end{equation}
the \emph{$(\ell)^-$-stabilization of $F$} as
\begin{equation}\label{eq:stabilizationmoduli-0}
F^{(\ell)^-} = F - \alpha_{\ell}(F)V_{\ell}^*F,
\end{equation}
and the \emph{$(\ell)^0$-stabilization for $F$} as
\begin{equation}\label{eq:stabilizationmoduli1}
F^{(\ell)^0} = F - a_{\ell}(F)V_{\ell}^*F + \ell^{k-1}V_{\ell}^*V_{\ell}^*F.
\end{equation}
We have $F^{(\ell)^{*}}\in M_k^{\text{$p$-adic}}(\Gamma_0(N^{\#}))$ for $* \in \{+,-,0\}$. On $q$-expansions, we have
$$F^{(\ell)^+}(q) := F(q) - \beta_{\ell}(F)F(q^{\ell}),$$
$$F^{(\ell)^-}(q) := F(q) - \alpha_{\ell}(F)F(q^{\ell}),$$
$$F^{(\ell)^0}(q) := F(q) - a_{\ell}(F)F(q^{\ell}) + \ell^{k-1}F(q^{\ell^2}).$$
It follows that if $F$ is a $T_n$-eigenform where $\ell\nmid n$, then $F^{(\ell)^{*}}$ is still an eigenform for $T_n$. If $F$ is a $T_{\ell}$-eigenform, one verifies by direct computation that $a_{\ell}(F^{(\ell)^+}) = \alpha_{\ell}(F)$, $a_{\ell}(F^{(\ell)^-}) = \beta_{\ell}(F)$, and $a_{\ell}(F^{(\ell)^0}) = 0$. 

When $\ell|N$, we define the \emph{$(\ell)^0$-stabilization of $F$} as
\begin{equation}\label{eq:stabilizationmoduli2}
F^{(\ell)^0} = F - a_{\ell}(F)V_{\ell}^*F.
\end{equation}
Again, we have $F^{(\ell)^0} \in M_k^{\text{$p$-adic}}(\Gamma_0(N^{\#}))$. On $q$-expansions, we have
$$F^{(\ell)^0}(q) := F(q) - a_{\ell}(F)F(q^{\ell}).$$
It follows that if $F$ is a $U_n$-eigenform where $\ell\nmid n$, then $F^{(\ell)^0}$ is still an eigenform for $U_n$. If $F$ is a $U_{\ell}$-eigenform, one verifies by direct computation that $a_{\ell}(F^{(\ell)^0}) = 0$. 

Note that for $\ell_1 \neq \ell_2$, the stabilization operators $F \mapsto F^{(\ell_1)^{*}}$ and $F\mapsto F^{(\ell_2)^{*}}$ commute. Then for pairwise coprime integers with prime factorizations $N_+ = \prod_i \ell_i^{e_i}$, $N_- = \prod_j \ell^{e_j}$, $N_0 = \prod_m \ell_m^{e_m}$, we define the \emph{$(N_+,N_-,N_0)$-stabilization of $F$} as
$$F^{(N_+,N_-,N_0)} := F^{\prod_i (\ell_i)^+\prod_j (\ell_j)^-\prod_m (\ell_m)^0}.$$ 
\end{definition}

\subsection{Stabilization operators at CM points}
Let $K$ be an imaginary quadratic field satisfying the Heegner hypothesis with respect to $N^{\#}$. Assume that $p$ splits in $K$, and let $\mathfrak{p}$ be prime above $p$ determined by the embedding $K \subset \mathbb{C}_p$. Let $\mathfrak{N}^{\#} \subset \mathcal{O}_K$ be a fixed ideal such that $\mathcal{O}/\mathfrak{N}^{\#} = \mathbb{Z}/N^{\#}$, and if $p|N^{\#}$, we assume that $\mathfrak{p}|\mathfrak{N}^{\#}$. Let $A/\mathcal{O}_{\mathbb{C}_p}$ be an elliptic curve with CM by $\mathcal{O}_K$. By the theory of complex multiplication and Deuring's theorem, $(A,A[\mathfrak{N}^{\#}],\omega)$ is an ordinary test triple over $\mathcal{O}_{\mathbb{C}_p}$.

Given an ideal $\mathfrak{a} \subset \mathcal{O}_K$, we define $A_{\mathfrak{a}} = A/A[\mathfrak{a}]$, an elliptic curve over $\mathcal{O}_{\mathbb{C}_p}$ which has CM by $\mathcal{O}_K$.
%If $(\alpha) \subset \mathcal{O}_K$ is a principal ideal which is prime to $\mathfrak{N}^{\#}$, the canonical isomorphism $(\alpha)\star A \xrightarrow{\sim} A$ also induces an isomorphism $(\alpha)\star (A,A[\mathfrak{N}^{\#}],\omega) \xrightarrow{\sim} (A,A[\mathfrak{N}^{\#}],\omega)$. 
Let $\phi_{\mathfrak{a}} : A \rightarrow A_{\mathfrak{a}}$ denote the canonical projection. Note that there is an induced action of prime-to-$\mathfrak{N}^{\#}$ integral ideals $\mathfrak{a}\subset\mathcal{O}_K$ on the set of triples $(A,A[\mathfrak{N}^{\#}],\omega)$ given by  of isomorphism classes $[(A,A[\mathfrak{N}^{\#}],\omega)]$, given by
$$\mathfrak{a}\star(A,A[\mathfrak{N}^{\#}],\omega) = (A_{\mathfrak{a}},A_{\mathfrak{a}}[\mathfrak{N}^{\#}],\omega_{\mathfrak{a}})$$
where $\omega_{\mathfrak{a}}\in\Omega_{A_{\mathfrak{a}}/\mathbb{C}_p}^1$ is the unique differential such that $\phi_{\mathfrak{a}}^*\omega_{\mathfrak{a}} = \omega$. Note that this action descends to an action on the set of isomorphism classes of triples $[(A,A[\mathfrak{N}^{\#}],\omega)]$ given by
$\mathfrak{a}\star[(A,A[\mathfrak{N}^{\#}],\omega)] = [\mathfrak{a}\star(A,A[\mathfrak{N}^{\#}],\omega)]$. Letting $\mathfrak{N} = (\mathfrak{N}^{\#},N)$, also note that for any $\mathfrak{N}'\subset \mathcal{O}_K$ with norm $N'$ and $\mathfrak{N}|\mathfrak{N}'|N^{\#}$, the Shimura reciprocity law also induces an action of prime-to-$\mathfrak{N}'$ integral ideals on CM test triples and isomorphism classes of ordinary CM test triples of level $N'$.

The following is Lemma 2.6 of \cite{KrizLi2016a}.

\begin{lemma}\label{lemma-calc}Suppose $F\in \tilde{M}_k^{\text{$p$-adic}}(\Gamma_0(N^{\#}))$, and let $\chi: \mathbb{A}_K^{\times} \rightarrow \mathbb{C}_p^{\times}$ be a $p$-adic Hecke character such $\chi$ is unramified (at all finite places of $K$), and $\chi_{\infty}(\alpha) = \alpha^k$ for any $\alpha \in K^{\times}$. Let $\{\mathfrak{a}\}$ be a full set of integral representatives of $\mathcal{C}\ell(\mathcal{O}_K)$ where each $\mathfrak{a}$ is prime to $\mathfrak{N}^{\#}$. If $\ell\nmid N$, we have
\begin{align*}\sum_{[\mathfrak{a}]\in\mathcal{C}\ell(\mathcal{O}_K)}\chi^{-1}(\mathfrak{a})&F^{(\ell)^+}(\mathfrak{a}\star(A,A[\mathfrak{N}^{\#}],\omega)) \\
&= \left(1-\beta_{\ell}(F)\chi^{-1}(\overline{v})\right)\sum_{[\mathfrak{a}]\in\mathcal{C}\ell(\mathcal{O}_K)}\chi^{-1}(\mathfrak{a})F(\mathfrak{a}\star(A,A[\mathfrak{N}^{\#}],\omega)),
\end{align*}
\begin{align*}\sum_{[\mathfrak{a}]\in\mathcal{C}\ell(\mathcal{O}_K)}\chi^{-1}(\mathfrak{a})&F^{(\ell)^-}(\mathfrak{a}\star(A,A[\mathfrak{N}^{\#}],\omega)) \\
&= \left(1-\alpha_{\ell}(F)\chi^{-1}(\overline{v})\right)\sum_{[\mathfrak{a}]\in\mathcal{C}\ell(\mathcal{O}_K)}\chi^{-1}(\mathfrak{a})F(\mathfrak{a}\star(A,A[\mathfrak{N}^{\#}],\omega)),
\end{align*}
\begin{align*}\sum_{[\mathfrak{a}]\in\mathcal{C}\ell(\mathcal{O}_K)}\chi^{-1}(\mathfrak{a})&F^{(\ell)^0}(\mathfrak{a}\star(A,A[\mathfrak{N}^{\#}],\omega))\\
&= \left(1-a_{\ell}(F)\chi^{-1}(\overline{v}) + \frac{\chi^{-2}(\overline{v})}{\ell}\right)\sum_{[\mathfrak{a}]\in\mathcal{C}\ell(\mathcal{O}_K)}\chi^{-1}(\mathfrak{a})F(\mathfrak{a}\star(A,A[\mathfrak{N}^{\#}],\omega))
\end{align*}
and if $\ell|N$, we have
\begin{align*}&\sum_{[\mathfrak{a}]\in\mathcal{C}\ell(\mathcal{O}_K)}\chi^{-1}(\mathfrak{a})F^{(\ell)^0}(\mathfrak{a}\star(A,A[\mathfrak{N}^{\#}],\omega)) \\
&= \left(1-a_{\ell}(F)\chi^{-1}(\overline{v})\right)\sum_{[\mathfrak{a}]\in\mathcal{C}\ell(\mathcal{O}_K)}\chi^{-1}(\mathfrak{a})F(\mathfrak{a}\star(A,A[\mathfrak{N}^{\#}],\omega)).
\end{align*}
\end{lemma}

\subsection{The Eisenstein congruence}

We may assume without loss of generality that $\psi \neq \omega$ (otherwise, interchange $\psi$ and $\psi^{-1}\omega$). As in the proof of Theorem 13 in \cite{Kriz2016}, the argument relies on establishing an Eisenstein congruence. More precisely, let $f \in S_2(\Gamma_0(N))$ be the normalized newform associated with $A$ by modularity, and let $A_f$ be the $\GL_2$-type abelian variety associated with $f$ by Eichler-Shimura theory, so that $A$ is isogenous with $A_f$ by Faltings' isogeny theorem. Also suppose (without loss of generality) that $A_f$ satisfies our assumptions stated just before Theorem \ref{thm:Heegnercorollary}. Recall the weight 2 Eisenstein series $E_{2,\psi}$ defined by the $q$-expansion (at $\infty$)
$$E_{2,\psi}(q) := \delta(\psi)\frac{L(-1,\psi)}{2} + \sum_{n=1}^{\infty}\sigma^{\psi,\psi^{-1}}(n)q^n$$
where $\delta(\psi) = 1$ if $\psi = 1$ and $\delta(\psi) = 0$ otherwise, and
$$\sigma^{\psi,\psi^{-1}}(n) = \sum_{0 < d|n}\psi(n/d)\psi^{-1}(d)d.$$
This determines a $\Gamma_0(f(\psi)^2)$-level algebraic modular form of weight 2, in Katz's sense (see \cite[Chapter II]{Katz1976}). %The assumption that $E[p]$ is reducible and $E[p]^{\mathrm{ss}} \cong \mathbb{F}_p(\psi)\oplus\mathbb{F}_p(\psi^{-1}\omega)$ implies the following lemma (see \cite[Theorem 34 (2)]{Kriz2016}).
%\begin{lemma}\label{decomposition}$N$ has a decomposition $N = N_+N_-N_0$ into pairwise coprime integers $N_+,N_-, N_0$ such that that $N_+N_-$ is the square-free part of $N$, $N_0$ is the square-full part of $N$, and 
%  \begin{enumerate}
%  \item if $\ell|N_+$, then $a_{\ell}(f) \equiv \psi(\ell) \pmod{p}$,
%  \item if $\ell|N_-$, then $a_{\ell}(f) \equiv \psi^{-1}(\ell)\ell \pmod{p}$,
%  \item if $\ell|N_0$, then $a_{\ell}(f) = 0$.
%  \end{enumerate}
%\end{lemma}

Note that the minimal level of $E_{2,\psi}$ is $f(\psi)^2$. With respect to this level, take $N^{\#}$ as in \S \ref{stabilization} to be $N^{\#} = \mathrm{lcm}_{\ell|N}(\ell^2,f(\psi))$. We now consider $E_{2,\psi}$ as a form of level $N^{\#}$ and let $E_{2,\psi}^{(N_+,N_-,N_0)}$ denote the $(N_+,N_-,N_0)$-stabilization of $E_{2,\psi}$, with the choices $\alpha_{\ell} = \psi(\ell)$ and $\beta_{\ell} = \psi^{-1}(\ell)\ell$ as in Definition \ref{stabilization}. . Thus, viewing $f$ and $E_{2,\psi}^{(N_+,N_-,N_0)}$ as a $p$-adic $\Gamma_0(N)$-level modular forms over $\mathcal{O}_{\mathbb{C}_p}$, we have
$$\theta^jf(q) \equiv \theta^jE_{2,\psi}^{(N_+,N_-,N_0)}(q) \pmod{\lambda\mathcal{O}_{\mathbb{C}_p}}$$
for all $j \ge 1$. 

Let $A_0$ be a fixed elliptic curve with complex multiplication by $\mathcal{O}_K$, and fix an ideal $\mathfrak{N} \subset \mathcal{O}_K$ such that $\mathcal{O}_K/\mathfrak{N} = \mathbb{Z}/N$ and $\mathfrak{p}|\mathfrak{N}$ if $p|N$. Since $p$ is split in $K$, the $q$-expansion principle implies that the above congruences of $q$-expansions translate to congruences on points corresponding to curves with CM by $\mathcal{O}_K$. Let $P_f \in A_f(K)$ denote the Heegner point associated with $A_f$. As is explained in \S 2 of \cite{KrizLi2016a}, by a generalization of Coleman's theorem (\cite[Proposition A.1]{Liu2014}, see also \cite[Theorem 2.8]{KrizLi2016a}), this implies that (for any generator $\omega \in \Omega_{A_0/\mathcal{O}_{\mathbb{C}_p}}^1$)
\begin{equation}\label{Eiscong}\begin{split}\frac{1+p-a_p}{p}\cdot\log_{\omega_{A}}P &= \frac{1+p-a_p}{p}\cdot\log_{\omega_{A_f}}P_f\\
&= \sum_{[\mathfrak{a}]\in\mathcal{C}\ell(\mathcal{O}_K)}\theta^{-1}f^{(1,1,p)}(\mathfrak{a}\star(A_0,A_0[\mathfrak{N}],\omega))\\
&\equiv \sum_{[\mathfrak{a}]\in\mathcal{C}\ell(\mathcal{O}_K)}\theta^{-1}E_{2,\psi}^{(N_+,N_-,pN_0)}(\mathfrak{a}\star(A_0,A_0[\mathfrak{N}],\omega))\\
&= \prod_{\ell|N_+,\ell\neq p}\left(1-\psi^{-1}(\ell)\right)\prod_{\ell|N_-,\ell\neq p}\left(1-\frac{\psi(\ell)}{\ell}\right)\prod_{\ell|N_0,\ell\neq p}\left(1-\psi^{-1}(\ell)\right)\left(1-\frac{\psi(\ell)}{\ell}\right)\\
&\hspace{4cm}\cdot\sum_{[\mathfrak{a}]\in\mathcal{C}\ell(\mathcal{O}_K)}\theta^{-1}E_{2,\psi}^{(1,1,p)}(\mathfrak{a}\star(A_0,A_0[\mathfrak{N}],\omega)) \pmod{\lambda\mathcal{O}_{\mathbb{C}_p}}
\end{split}
\end{equation}
where the final equality follows from Lemma \ref{lemma-calc}, applied to successive stabilizations of $E_{2,\psi}$. 

\subsection{CM period of Eisenstein series} To evaluate (\ref{Eiscong}) further, we need to study the period
$$\sum_{[\mathfrak{a}]\in\mathcal{C}\ell(\mathcal{O}_K)}\theta^{-1}E_{2,\psi}^{(1,1,p)}(\mathfrak{a}\star(A_0,A_0[\mathfrak{N}],\omega)) \pmod{\lambda\mathcal{O}_{\mathbb{C}_p}}.$$
We will show that this period is interpolated by the Katz $p$-adic $L$-function. Indeed, let $\chi_j$ be the unramified Hecke character of infinity type $(h_Kj,-h_Kj)$ defined on ideals by
$$\chi_j(\mathfrak{a}) = (\alpha/\overline{\alpha})^j$$
where $(\alpha) = \mathfrak{a}^{h_K}$, and $h_K$ is the class number of $K$. Let $\overline{\mathfrak{p}}$ denote the prime ideal of $\mathcal{O}_K$ which is the complex conjugate of $\mathfrak{p}$. For the remainder of the proof, in a slight abuse of notation, unless otherwise stated let $\mathbb{N}_K$ denote the $p$-adic Hecke character associated with the algebraic Hecke character giving rise to the complex Hecke character $\mathbb{N}_K : K^{\times}\backslash\mathbb{A}_K^{\times} \rightarrow \mathbb{C}^{\times}$. Then by looking at $q$-expansions and invoking the $q$-expansion principle, it is apparent that the above sum is given by
\begin{equation}\label{limit}\begin{split}&\sum_{[\mathfrak{a}]\in\mathcal{C}\ell(\mathcal{O}_K)}\theta^{-1}E_{2,\psi}^{(1,1,p)}(\mathfrak{a}\star(A_0,A_0[\mathfrak{N}],\omega)) \\
&= \lim_{j \rightarrow 0}\sum_{[\mathfrak{a}]\in\mathcal{C}\ell(\mathcal{O}_K)}(\chi_j^{-1}\mathbb{N}_K^{h_Kj})(\mathfrak{a})\theta^{-1+h_Kj}E_{2,\psi}^{(1,1,p)}(\mathfrak{a}\star(A_0,A_0[\mathfrak{N}],\omega))\\
&= \lim_{j \rightarrow 0}(1-\psi^{-1}(p)\chi_j^{-1}(\overline{\mathfrak{p}}))(1-\psi(p)(\chi_j^{-1}\mathbb{N}_K)(\overline{\mathfrak{p}}))\\
&\hspace{3.1cm}\cdot\sum_{[\mathfrak{a}]\in\mathcal{C}\ell(\mathcal{O}_K)}(\chi_j^{-1}\mathbb{N}_K^{h_Kj})(\mathfrak{a})\theta^{-1+h_Kj}E_{2,\psi}(\mathfrak{a}\star(A_0,A_0[\mathfrak{N}],\omega))
\end{split}
\end{equation}
since $\chi_j^{-1}\mathbb{N}_K^{h_Kj}\rightarrow 1$ as $j \rightarrow 0 = (0,0) \in \mathbb{Z}/(p-1)\times\mathbb{Z}_p$; here the last equality again follows from Lemma \ref{lemma-calc} applied to $F = E_{2,\psi}$. 

\subsection{The Katz $p$-adic $L$-function} We will now show that the terms in the above limit are interpolated by the Katz $p$-adic $L$-function (restricted to the anticyclotomic line). Let $\mathfrak{f}|\mathfrak{N}$ such that $\mathcal{O}/\mathfrak{f} =  \mathbb{Z}/f(\psi)$.  Choose a good integral model $\mathcal{A}_0$ of $A_0$ at $p$, choose an identification $\iota : \hat{\mathcal{A}}_0 \xrightarrow{\sim} \hat{\mathbb{G}}_m$ (unique up to $\mathbb{Z}_p^{\times}$), and let $\omega_{\mathrm{can}} := \iota^*\frac{du}{u}$ where $u$ is the coordinate on $\hat{\mathbb{G}}_m$. This choice of $\omega_{\mathrm{can}}$ determines $p$-adic and complex periods $\Omega_p$ and $\Omega_{\infty}$ as in Section 3 of \cite{Kriz2016}. As an intermediate step to establishing the $p$-adic interpolation, we have the following identity of algebraic values.
\begin{lemma}\label{algvals}We have the following identity of values in $\overline{\mathbb{Q}}$ for $j \ge 1$:
\begin{align*}&\sum_{[\mathfrak{a}]\in \mathcal{C}\ell(\mathcal{O}_K)}(\chi_j^{-1}\mathbb{N}_K^{h_Kj})(\mathfrak{a})\theta^{-1+h_Kj}E_{2,\psi}(\mathfrak{a}\star(A_0,A_0[\mathfrak{N}],\omega_{\mathrm{can}})) \\
&=  \left(\frac{\Omega_p}{\Omega_{\infty}}\right)^{2h_Kj}\cdot \frac{f(\psi)^2\Gamma(1+h_kj)\psi^{-1}(-\sqrt{d_K})(\chi_j^{-1}\mathbb{N}_K)(\overline{\mathfrak{f}})}{(2\pi i)^{1+h_Kj}\mathfrak{g}(\psi^{-1})(\sqrt{d_K})^{-1+h_Kj}}L((\psi\circ\Nm_{K/\mathbb{Q}})\chi_j^{-1}\mathbb{N}_K,0)
\end{align*}
where $\psi^{-1}(-\sqrt{d_K})$ denotes the Dirichlet character $\psi^{-1}$ evaluated at the unique class $b \in (\mathbb{Z}/f(\psi))^{\times}$ such that $b + \sqrt{d_K} \equiv 0 \pmod{\mathfrak{f}}$. (In particular, note that the above complex-analytic calculation does not use the assumptions $p > 2$ or $p\nmid f(\psi)$.) 
\end{lemma}
\begin{proof}View the algebraic modular form $E_{2,\psi}$ as a modular form over $\mathbb{C}$, and evaluate at CM triples $(A_0,A_0[\mathfrak{N}],2\pi i dz)$ as a triple over $\mathbb{C}$ by considering the uniquely determined complex uniformization $\mathbb{C}/(\mathbb{Z}\tau + \mathbb{Z})\cong A_0$ for some $\tau$ in the complex upper half-plane, and identifying $A_0[\mathfrak{N}]$ with $\frac{1}{N}\mathbb{Z} \subset \mathbb{C}/(\mathbb{Z}\tau + \mathbb{Z})$. By plugging $\psi_1 = \psi_2^{-1} = \psi$ and $\mathfrak{u} = \mathfrak{t} = \mathfrak{f}$, $\mathfrak{N}' = \mathfrak{f}^2$ into Proposition 36 of loc. cit., we have the complex identity
\begin{equation}\label{complexid}\begin{split}&\sum_{[\mathfrak{a}]\in \mathcal{C}\ell(\mathcal{O}_K)}(\chi_j^{-1}\mathbb{N}_K^{h_Kj})(\mathfrak{a})\partial^{-1+h_Kj}E_{2,\psi}(\mathfrak{a}\star(A_0,A_0[\mathfrak{N}],2\pi i dz)) \\
&=  \frac{f(\psi)^2\Gamma(1+h_kj)\psi^{-1}(-\sqrt{d_K})(\chi_j^{-1}\mathbb{N}_K)(\overline{\mathfrak{f}})}{(2\pi i)^{1+h_Kj}\mathfrak{g}(\psi^{-1})(\sqrt{d_K})^{-1+h_Kj}}L((\psi\circ\Nm_{K/\mathbb{Q}})\chi_j^{-1}\mathbb{N}_K,0)
\end{split}
\end{equation}
where $\partial$ is the complex Maass-Shimura operator, and $\mathbb{N}_K : K^{\times}\backslash \mathbb{A}_K^{\times} \rightarrow \mathbb{C}^{\times}$ is the complex norm character over $K$. By definition of $\Omega_p$ and $\Omega_{\infty}$, we have
$$2\pi idz = \frac{\Omega_p}{\Omega_{\infty}}\cdot \omega_{\mathrm{can}}.$$
By Proposition 21 of loc. cit., we have the equality of \emph{algebraic} values
$$\partial^{-1+h_kj}E_2(\mathfrak{a}\star(A_0,A_0[\mathfrak{N}],\omega_{\mathrm{can}})) = \theta^{-1+h_kj}E_2(\mathfrak{a}\star(A_0,A_0[\mathfrak{N}],\omega_{\mathrm{can}}))$$
for all $j \ge 1$. Moreover, since $\mathbb{N}_K(\mathfrak{a}) \in \overline{\mathbb{Z}}$, we can identify this value of $\mathbb{N}_K$ with the value of its $p$-adic avatar, which again we also denote by $\mathbb{N}_K$, at $\mathfrak{a}$. Applying these identities to the identity of complex numbers (\ref{complexid}), we get the desired identity of algebraic numbers.
\end{proof}

We now apply the interpolation property of the Katz $p$-adic $L$-function (see \cite[Theorem II]{HidaTilouine1993}) to our situation, taking the normalization as in \cite{Gross1980}, thus arriving at the identity
%, under a fixed isomorphism $i : \mathbb{C} \xrightarrow{\sim} \mathbb{C}_p$ (which identifies a complex algebraic Hecke character $\chi$ with its $p$-adic avatar),
\begin{equation}\label{Katzinterpolation}\begin{split}&L_p^{\mathrm{Katz}}((\psi\circ\Nm_{K/\mathbb{Q}})\chi_j^{-1}\mathbb{N}_K,0) = 4\cdot\mathrm{Local}_{\mathfrak{p}}((\psi\circ\Nm_{K/\mathbb{Q}})\chi_j^{-1}\mathbb{N}_K)\left(\frac{\Omega_p}{\Omega_{\infty}}\right)^{2h_Kj}\\
&\cdot\left(\frac{2\pi i}{\sqrt{D_K}}\right)^{-1+h_Kj}\Gamma(1+h_Kj)(1-\psi(p)(\chi_j^{-1}\mathbb{N}_K)(\overline{\mathfrak{p}}))(1-\psi(p)\chi_j^{-1}(\overline{\mathfrak{p}}))L((\psi\circ\Nm_{K/\mathbb{Q}})\chi_j^{-1}\mathbb{N}_K,0)
\end{split}
\end{equation}
for all $j \ge 1$, where $\mathrm{Local}_{\mathfrak{p}}(\chi) = \mathrm{Local}_{\mathfrak{p}}(\chi,\Sigma,\delta)$ is defined as in \cite[5.2.26]{Katz1978} with $\Sigma = \{\mathfrak{p}\}$ and $\delta = \sqrt{d_K}/2$ (or as denoted $W_p(\lambda)$ in \cite[0.10]{HidaTilouine1993}). For any prime $\ell$, let $\psi_{\ell}(-\sqrt{d_K})$ denote the value $\psi_{\ell}(b)$, where again $b \in \mathbb{Z}$ is any integer such that $b+\sqrt{d_K} \in \mathfrak{f}$. By directly plugging in $\chi = (\psi\circ\Nm_{K/\mathbb{Q}})\chi_j^{-1}\mathbb{N}_K$ into the definition of $\mathrm{Local}_{\mathfrak{p}}$, we have
$$\mathrm{Local}_{\mathfrak{p}}((\psi\circ\Nm_{K/\mathbb{Q}})\chi_j^{-1}\mathbb{N}_K) = \psi_p(-\sqrt{d_K})\frac{f(\psi)_p}{\mathfrak{g}_p(\psi)}.$$
Plugging (\ref{Katzinterpolation}) into the identity in Lemma \ref{algvals}, we have for all $j \ge 1$
\begin{align*}(1-\psi^{-1}(p)\chi_j^{-1}(\overline{\mathfrak{p}}))(1-&\psi(p)(\chi_j^{-1}\mathbb{N}_K)(\overline{\mathfrak{p}}))\sum_{[\mathfrak{a}]\in\mathcal{C}\ell(\mathcal{O}_K)}(\chi_j^{-1}\mathbb{N}_K^{h_Kj})(\mathfrak{a})\theta^{-1+h_Kj}E_{2,\psi}(\mathfrak{a}\star(A_0,A_0[\mathfrak{N}],\omega_{\mathrm{can}})) \\
&= \frac{f(\psi)^{(p)}\cdot f(\psi)\cdot(\chi_j^{-1}\mathbb{N}_K)(\overline{\mathfrak{f}})}{4(\prod_{\ell|f(\psi)^{(p)}}\psi_{\ell}^{-1}(-\sqrt{d_K})\mathfrak{g}_{\ell}(\psi))(2\pi i)^{2h_Kj}}L_p^{\mathrm{Katz}}((\psi\circ\Nm_{K/\mathbb{Q}})\chi_j^{-1}\mathbb{N}_K,0).
\end{align*}
Taking the limit $j \rightarrow 0 = (0,0) \in \mathbb{Z}/(p-1)\times\mathbb{Z}_p$, noting that $\chi_j^{-1}\mathbb{N}_K \rightarrow \mathbb{N}_K$ and $\mathbb{N}_K(\overline{\mathfrak{f}}) = f(\psi)^{-1}$, and applying (\ref{limit}), we have
\begin{equation}\label{Katzperiod}\sum_{[\mathfrak{a}]\in\mathcal{C}\ell(\mathcal{O}_K)}\theta^{-1}E_{2,\psi}^{(1,1,p)}(\mathfrak{a}\star(A_0,A_0[\mathfrak{N}],\omega_{\mathrm{can}})) = \frac{f(\psi)^{(p)}}{4(\prod_{\ell|f(\psi)^{(p)}}\psi_{\ell}^{-1}(-\sqrt{d_K})\mathfrak{g}_{\ell}(\psi))}L_p^{\mathrm{Katz}}((\psi\circ\Nm_{K/\mathbb{Q}})\mathbb{N}_K,0).
\end{equation}

\subsection{Gross's factorization theorem} We now evaluate the Katz $p$-adic $L$-value on the right-hand side of (\ref{Katzperiod}).

\begin{lemma}\label{factorization}We have, for $\psi \neq 1$,
\begin{align*}\sum_{[\mathfrak{a}]\in\mathcal{C}\ell(\mathcal{O}_K)}\theta^{-1}E_{2,\psi}^{(1,1,p)}(\mathfrak{a}\star(A_0,A_0[\mathfrak{N}],\omega_{\mathrm{can}})) &\\
&\hspace{-2cm}= \pm\frac{1}{4}(1-\psi^{-1}(p))(1-(\psi\omega^{-1})(p))B_{1,\psi_0^{-1}\varepsilon_K}B_{1,\psi_0\omega^{-1}} \pmod{p\mathcal{O}_{\mathbb{C}_p}}
\end{align*}
and for $\psi = 1$,
\begin{align*}\sum_{[\mathfrak{a}]\in\mathcal{C}\ell(\mathcal{O}_K)}\theta^{-1}E_{2,1}^{(1,1,p)}(\mathfrak{a}\star(A_0,A_0[\mathfrak{N}],\omega_{\mathrm{can}})) \equiv \frac{p-1}{2p}\log_p \overline{\alpha} \pmod{p\mathcal{O}_{\mathbb{C}_p}}
\end{align*}
where $\alpha \in \mathcal{O}_K$ such that $(\alpha) = \mathfrak{p}^{h_K}$.
\end{lemma}
\begin{proof}
Applying Gross's factorization theorem (see \cite{Gross1980}, and \cite[Theorem 28]{Kriz2016} for the extension to the general auxiliary conductor case), we have
\begin{equation}\label{id0}\frac{f(\psi)^{(p)}}{(\prod_{\ell|f(\psi)^{(p)}}\psi_{\ell}^{-1}(-\sqrt{d_K})\mathfrak{g}_{\ell}(\psi))}L_p^{\mathrm{Katz}}((\psi\circ\Nm_{K/\mathbb{Q}})\mathbb{N}_K,0) = \pm L_p(\psi_0^{-1}\varepsilon_K\omega,0)L_p(\psi_0,1)
\end{equation}
where $L_p(\cdot,s)$ denotes the Kubota-Leopoldt $p$-adic $L$-function; here the sign of $\pm 1$ is uniquely determined, as in loc. cit., by the special value formula due to Katz used in Gross's proof (the term on the left-hand side of the statement in loc. cit. already incorporates this sign). We now evaluate each Kubota-Leopoldt factor in the above identity. Using the fact that $\varepsilon_K(p) = 1$ since $p$ splits in $K$, by the interpolation property of the Kubota-Leopoldt $p$-adic $L$-function we have
\begin{equation}\label{id1}L_p(\psi_0^{-1}\varepsilon_K,0)=  -(1-\psi^{-1}(p))B_{1,\psi_0^{-1}\varepsilon_K}.
\end{equation}

Now if by condition (3) in the case $\psi \neq 1$ of the statement of the theorem, we know that $L_p(\psi_0,m) \equiv L_p(\psi_0,n) \pmod{p\mathcal{O}_{\mathbb{C}_p}}$ for all $m,n\in \mathbb{Z}$ by \cite[Corollary 5.13]{Washington1997}. Thus
\begin{equation}\label{id2}
L_p(\psi_0,1) \equiv L_p(\psi_0,0) = -(1-(\psi\omega^{-1})(p))B_{1,\psi_0\omega^{-1}} \pmod{p\mathcal{O}_{\mathbb{C}_p}}.
\end{equation}

Combining (\ref{id0}), (\ref{id1}), and (\ref{id2}), we get 
\begin{equation}\begin{split}\label{idnontriv}&\frac{f(\psi)^{(p)}}{\prod_{\ell|f(\psi)^{(p)}}\psi_{\ell}^{-1}(-\sqrt{d_K})\mathfrak{g}_{\ell}(\psi)}L_p^{\mathrm{Katz}}((\psi\circ\Nm_{K/\mathbb{Q}})\mathbb{N}_K,0)\\
&\hspace{4.2cm} \equiv \pm(1-\psi^{-1}(p))(1-(\psi\omega^{-1})(p))B_{1,\psi_0^{-1}\varepsilon_K}B_{1,\psi_0\omega^{-1}} \pmod{p\mathcal{O}_{\mathbb{C}_p}}
\end{split}
\end{equation}
when $\psi \neq 1$.

Now suppose $\psi = 1$. In particular $f(\psi) = f(\psi)^{(p)} = 1$. By the functional equation for the Katz $p$-adic $L$-function (e.g. see \cite[Theorem II]{HidaTilouine1993}), since $\check{\mathbb{N}}_K = \mathbb{N}_K^{-1}\mathbb{N}_K = 1$ is the dual Hecke character of $\mathbb{N}_K$, we have
$$L_p^{\mathrm{Katz}}(\mathbb{N}_K,0) = L_p^{\mathrm{Katz}}(1,0).$$
By a standard special value formula (e.g. see \cite[Section 5, Formulas 1]{Gross1980}), we have
$$L_p^{\mathrm{Katz}}(1,0) = \frac{4}{|\mathcal{O}_K^{\times}|}\cdot\frac{p-1}{p}\log_p(\overline{\alpha})$$
and so
\begin{equation}\label{idtriv}L_p^{\mathrm{Katz}}(\mathbb{N}_K,0) = \frac{4}{|\mathcal{O}_K^{\times}|}\cdot\frac{p-1}{p}\log_p(\overline{\alpha}) = 2\cdot\frac{p-1}{p} \log_p(\overline{\alpha})
\end{equation}
since we assume $d_K < -4$ and hence $|\mathcal{O}_K^{\times}| = 2$.

Now plugging in (\ref{idnontriv}) into (\ref{Katzperiod}) when $\psi \neq 1$, and (\ref{idtriv}) into (\ref{Katzperiod}) when $\psi = 1$, we establish the lemma.
\end{proof}

\subsection{The proof of Theorem \ref{thm:Heegnercorollary}}\label{maincongruence} Putting together (\ref{Eiscong}) and Lemma (\ref{factorization}), we arrive at our main congruence identities. If $\psi \neq 1$ we have
\begin{equation}\label{nontrivcong}\begin{split}\frac{1+p-a_p}{p}\cdot\log_{\omega_A}P \equiv \pm\prod_{\ell|N_+,\ell\neq p}\left(1-\psi^{-1}(\ell)\right)\prod_{\ell|N_-,\ell\neq p}\left(1-\frac{\psi(\ell)}{\ell}\right)\prod_{\ell|N_0,\ell\neq p}\left(1-\psi^{-1}(\ell)\right)\left(1-\frac{\psi(\ell)}{\ell}\right)\\
\cdot \frac{1}{4}(1-\psi^{-1}(p))(1-(\psi\omega^{-1})(p))B_{1,\psi_0^{-1}\varepsilon_K}B_{1,\psi_0\omega^{-1}} \pmod{\lambda\mathcal{O}_{\mathbb{C}_p}}.\end{split}
\end{equation}
Now the statement for $\psi \neq 1$ in theorem \ref{thm:Heegnercorollary} immediately follows from studying when the right-hand side of the congruence vanishes mod $p$. If $\psi = 1$ we have
\begin{equation}\label{trivcong}\frac{1+p-a_p}{p}\cdot\log_{\omega_A}P \equiv \begin{cases}\prod_{\ell|N_-,\ell\neq p}\left(1-\frac{1}{\ell}\right)\cdot\frac{p-1}{2p}\log_p \overline{\alpha}\pmod{\lambda\mathcal{O}_{\mathbb{C}_p}},&\text{if}\; \ell|N_+N_0 \implies \ell = p,\\
0 \hspace{4.69cm} \pmod{\lambda\mathcal{O}_{\mathbb{C}_p}},&\text{if}\; \exists \ell \neq p$ such that $\ell|N_+N_0,
\end{cases}
\end{equation}
where $(\overline{\alpha}) = \overline{\mathfrak{p}}^{h_K}$ and $\log_p$ is the Iwasawa $p$-adic logarithm (i.e. the locally analytic function defined by the usual power series $\log(1+x) = x - \frac{x^2}{2} + \frac{x^3}{3} - \ldots$, and then uniquely extended to all of $\mathbb{C}_p^{\times}$ by defining $\log_p p = 0$).

We now finish the proof of Theorem \ref{thm:Heegnercorollary} with the following lemma.

\begin{lemma}The right-hand side of (\ref{trivcong}) does not vanish mod $p$ if any only if 
\begin{enumerate}
\item $\ell|N, \ell \neq p$ implies $\ell||N, \ell \equiv -1 \pmod{p}, \ell \not\equiv 1 \pmod{p}$,
\item $\ord_{\lambda}\left(\frac{p-1}{2p}\log_p\overline{\alpha}\right) = 0$..
\end{enumerate}
We also have that the non-vanishing of the right-hand side of (\ref{trivcong}) mod $p$ implies $p|N$, and so the right-hand side of (\ref{trivcong}) does not vanish mod $p$ if and only if $p|N$ and (1) and (2) hold.
\end{lemma}

\begin{proof} 
We first study when 
\begin{equation}\label{quantity}\prod_{\ell|N_-,\ell \neq p}\left(1-\frac{1}{\ell}\right)\cdot \frac{p-1}{2p}\log_p\overline{\alpha}
\end{equation}
vanishes mod $\lambda$. Clearly (\ref{quantity}) does not vanish mod $\lambda$ if and only if each of its factors does not vanish mod $\lambda$. Then $\prod_{\ell|N_-,\ell \neq p}\left(1-\frac{1}{\ell}\right)$ does not vanish mod $\lambda$ if and only if 
\begin{equation}\label{implication}\ell|N_-, \ell \neq p \implies \ell \not\equiv 1 \pmod{p}.
\end{equation}
Hence (\ref{quantity}) does not vanish mod $p$ if and only if (\ref{implication}) and (2) in the statement of the lemma hold. 

If the right-hand side of (\ref{trivcong}) does not vanish, then we have $\ell|N_+N_ 0\implies \ell = p$, the right-hand side of (\ref{trivcong}) equals (\ref{quantity}) mod $p$, and (\ref{implication}) holds. Thus (1) and (2) in the statement of the lemma hold. 

If (1) and (2) in the statement of the lemma hold, then since by definition $\ell|N_- \implies \ell \equiv \pm 1 \pmod{p}$, we have that (\ref{implication}) holds. So (\ref{quantity}) does not vanish mod $p$. Now if $\ell|N_+N_0$ and $\ell \neq p$, then by (1) in the statement of the lemma, we have $\ell||N, \ell \not\equiv 1\pmod {p}$. Hence $\ell\nmid N_0, \ell \nmid N_+$, a contradiction. So we have $\ell|N_+N_0 \implies \ell = p$, and so the right-hand side of (\ref{trivcong}) equals (\ref{quantity}) mod $p$, which does not vanish mod $p$.

Thus we have shown that the non-vanishing of the right-hand side of (\ref{trivcong}) mod $p$ is equivalent to (1) and (2) in the statement of the lemma.

Now we show the second part of the theorem. Suppose that the right-hand side of (\ref{trivcong}) does not vanish. In particular, we have $\ell|N_+N_0\implies \ell = p$ and that the right-hand side of (\ref{trivcong}) equals (\ref{quantity}) mod $p$. If $p \nmid N$, then we thus have $N_+N_0 = 1$. We now show a contradiction, considering the cases $p = 2$ and $p \ge 3$ separately.

Suppose $p = 2$. Then since $2\nmid N_- = N \neq 1$ (where $N \neq 1$ follows because $E$ is an elliptic curve over $\mathbb{Q}$), we have that there exists $\ell|N_-$ with $\ell \equiv 1 \pmod{2}$. Hence
\begin{equation}\label{cong0}\prod_{\ell|N_-, \ell \neq p}\left(1-\frac{1}{\ell}\right) \equiv 0 \pmod{p}
\end{equation}
and the right-hand side of (\ref{trivcong}) vanishes mod $p$, a contradiction.

Suppose $p > 2$. Note that 
\begin{equation}\label{cong1}(N_{\splt},N_-) = \prod_{\ell|N_-, \ell \equiv 1 \pmod{p}}\ell.
\end{equation}
Since $N_0 = N_{\add}$ (because they are both the squarefull parts of $N$), we have $N_{\add} = N_0 = 1$. By \cite[Theorem 2.2]{Yoo2015}, we know that $N_{\splt}N_{\add} \neq 1$, and hence $N_{\splt} \neq 1$. Since $N_+ = 1$, we therefore have that $1 \neq N_{\splt}|N_-$. By (\ref{cong1}), we thus have that there is some $\ell|N_-$ such that $\ell \equiv 1 \pmod{p}$. In particular we have (\ref{cong0}) once again, and so the right-hand side of (\ref{trivcong}) vanishes mod $p$, a contradiction.
\end{proof}

\begin{remark}Note that our proof uses a direct method of $p$-adic integration, and does not go through the construction of the Bertolini--Darmon--Prasanna (BDP) $p$-adic $L$-function as in the proof of the main theorem of loc. cit. In particular, it does not recover the more general congruence of the BDP and Katz $p$-adic $L$-functions established when $p$ is of good reduction established in \cite{Kriz2016} (also for higher weight newforms). We expect that our method should extend to higher weight newforms, in particular establishing congruences between images of generalized Heegner cycles under appropriate $p$-adic Abel-Jacobi images and quantities involving higher Bernoulli numbers and Euler factors, without using the deep BDP formula.
\end{remark}

\section{Bernoulli numbers and relative class numbers}
\label{sec:bern-numb-relat}

When $p = 3$ and $A$ has a 3-isogeny defined over $\mathbb{Q}$, all Dirichlet characters in Theorem \ref{thm:Heegnercorollary} are quadratic. Note that for an odd quadratic character $\psi$ over $\mathbb{Q}$, by the analytic class number formula we have
\begin{equation}\label{eq:classnumberformula}
B_{1,\psi} = -2\frac{h_{K_{\psi}}}{|\mathcal{O}_{K_{\psi}}^{\times}|}  
\end{equation}
where $K_{\psi}$ is the imaginary quadratic field associated with $\psi$. 
So the $3$-indivisibility criteria of the theorem becomes a question of $3$-indivisibility of quadratic class numbers. This fact will be employed in our applications to Goldfeld's conjecture.

More generally, for $p \ge 3$, we can find a sufficient condition for non-vanishing mod $p$ of the Bernoulli numbers $B_{1,\psi_0^{-1}\varepsilon_K}B_{1,\psi_0\omega^{-1}}$ in terms of non-vanishing mod $p$ of the relative class numbers of the abelian CM fields of degrees dividing $p-1$ cut out by $\psi_0^{-1}\varepsilon_K$ and $\psi_0\omega^{-1}$. Let us first observe the following simple lemma.

%Let $\Phi : \mathbb{Z}_{>0} \rightarrow \mathbb{Z}_{>0}$ denote the Euler totient function, and for a ring $R$ and a Dirichlet character $\psi$, let $R[\psi]$ denote the finite extension of $R$ generated by the values of $\psi$. Let $\psi \pmod{p}$ denote the primitive Dirichlet character equivalent to  the composition of $\psi : (\mathbb{Z}/f)^{\times} \rightarrow \overline{\mathbb{Q}}$ with $\mathbb{Z}[\psi] \rightarrow \mathbb{Z}[\psi]/p\mathbb{Z}[\psi]$.
%Let
%$$\chi_f : \Gal(\mathbb{Q}(\mu_f)/\mathbb{Q}) \xrightarrow{\sim} \mu_{\Phi(f)}$$
%denote the mod $f$ cyclotomic character, i.e. the character which is defined via the relation
%$$\sigma\cdot\zeta = \zeta^{\chi_f(\sigma)}$$
%for $\sigma \in \Gal(\mathbb{Q}(\mu_f)/\mathbb{Q})$ and $\zeta \in \mu_{\Phi(f)}$.

\begin{lemma}\label{Bernoullidivisibility}Suppose $\psi : (\mathbb{Z}/f)^{\times} \rightarrow \mu_{p-1}$ is a Dirichlet character, and assume $\psi^{-1} \neq \omega$, or equivalently, assume there exists some $a \in (\mathbb{Z}/f)^{\times}$ such that $\psi(a)a \not\equiv 1 \pmod{p\mathbb{Z}[\mu_{p-1}]}$. Then
$$\ord_p(B_{1,\psi}) \ge 0.$$
\end{lemma}
\begin{proof}% Recall the formula
% \begin{equation}\label{Bernoulliformula} B_{1,\psi} = \frac{1}{f}\cdot \sum_{m = 1}^f\psi(m)m
% \end{equation}
% (see, for example, \cite[Theorem 4.17]{Washington1997}). 
By our assumption, there exists some $a \in (\mathbb{Z}/f)^{\times}$ such that $\psi(a)a \not\equiv 1 \pmod{p\mathbb{Z}[\mu_{p-1}]}$. Then we have
\begin{align*}&\sum_{m = 1}^f\psi(m)m \equiv \sum_{m = 1}^f\psi(am)am = \psi(a)a\sum_{m = 1}^f\psi(m)m \pmod{p\mathbb{Z}[\mu_{p-1}]} \\
&\implies (1-\psi(a)a)\cdot \sum_{m = 1}^f\psi(m)m \equiv 0 \pmod{p\overline{\mathbb{Z}}} \implies \sum_{m = 1}^f\psi(m)m \equiv 0 \pmod{p\mathbb{Z}[\mu_{p-1}]}.
\end{align*}
Now our conclusion follows from the formula for the Bernoulli numbers (\ref{Bernoulliformula}).
\end{proof}
For an odd Dirichlet character $\psi$, let $K_{\psi}$ denote the abelian CM field cut out by $\psi$. Consider the relative class number $h_{K_\psi}^-=h_{K_{\psi}}/h_{K_{\psi}^+}$, where $K_{\psi}^+$ is the maximal totally real subfield of $K_{\psi}$. The relative class number formula (\cite[4.17]{Washington1997}) gives 
\begin{equation}\label{relativeclassnoformula} h_{K_{\psi}}^- = Q\cdot w\cdot\prod_{\chi\text{odd}}\left(-\frac{1}{2}B_{1,\chi}\right)
\end{equation}
where $\chi$ runs over all odd characters of $\Gal(K_\psi/\mathbb{Q})$, $w$ is the number of roots of unity in $K_\psi$ and $Q = 1$ or $2$ (see \cite[4.12]{Washington1997}). % (the index of the group of units of $K_{\psi}^+$ modulo $\{\pm 1\}$ in the group of units of $K_{\psi}$ modulo $\mu(K_{\psi})$).
By Lemma \ref{Bernoullidivisibility}, assuming that $\psi^{-1} \not\equiv \omega$, we see that we have the following divisibility of numbers in $\mathbb{Z}_p[\psi]$:
\begin{equation}\label{classdiv}p\nmid h_{K_{\psi}}^- \implies p\nmid B_{1,\psi}.
\end{equation}

\begin{lemma}\label{reducetoclassno}Suppose $\psi : \Gal(\overline{\mathbb{Q}}/\mathbb{Q}) \rightarrow \mu_{p-1}$ is a Dirichlet character and $K$ is an imaginary quadratic field such that $f(\psi)$ is prime to $d_K$ and $p \nmid d_K$. As long as $\psi \neq 1$ or $\omega$, we have
$$p\nmid h_{K_{\psi_0\varepsilon_K}}^-\cdot h_{K_{\psi_0^{-1}\omega}}^- \implies p\nmid B_{1,\psi_0\varepsilon_K}\cdot B_{1,\psi_0^{-1}\omega}.$$
\end{lemma}
\begin{proof}If $\psi$ is even, then $\psi_0\varepsilon_K = \psi\varepsilon_K$ is ramified at some place outside $p$ and so is not equal to $\omega$, and $\psi_0^{-1}\omega = \psi^{-1}\omega$ is not equal to $\omega$ if and only if $\psi \neq 1$. Hence $(\psi_0^{-1}\varepsilon_K)^{-1} \pmod{p} = \psi_0\varepsilon_K \neq \omega$, and $(\psi_0\omega^{-1})^{-1} = \psi^{-1}\omega \neq \omega$ if and only if $\psi \neq 1$. If $\psi$ is odd, then $\psi_0\varepsilon_K = \psi$ is not equal to $\omega$ if and only if $\psi \neq \omega$, and $\psi_0^{-1}\omega = \psi^{-1}\varepsilon_K\omega$ is ramified at some place outside $p$ and so is not equal to $\omega$. Hence $(\psi_0^{-1}\varepsilon_K)^{-1} = \psi_0\varepsilon_K \neq \omega$ unless $\psi = \omega$, and $(\psi_0\omega^{-1})^{-1} = \psi^{-1}\varepsilon_K\omega \neq \omega$.

Now the lemma follows from (\ref{classdiv}).
\end{proof}

\begin{corollary}\label{cor:classnumberandbernoulli}Suppose we are in the setting of Theorem \ref{thm:Heegnercorollary}. Then $p\nmid h_{K_{\psi_0\varepsilon_K}}^-\cdot h_{K_{\psi_0^{-1}\omega}}^-$ implies condition (4) of the theorem.
\end{corollary}

\begin{proof}Condition (1) in the statement of Theorem \ref{thm:Heegnercorollary} in particular implies $\psi \neq 1$ or $\omega$. Now the statement follows from Lemma (\ref{reducetoclassno}).
\end{proof}

\section{Goldfeld's conjecture for abelian varieties over $\mathbb{Q}$ of $\GL_2$-type with a rational 3-isogeny}
\label{sec:goldf-conj-ellipt}

% Given $E/\mathbb{Q}$ and a fundamental discriminant $d$, let $E^{(d)}$ denote the quadratic twist by $d$. Let 
% $$N_r(E,X) := \#\{|d|\le X, d \hspace{.2cm} \text{fundamental discriminant} : \text{ord}_{s=1}L(E^{(d)},s) = r\}.$$
% Recall Goldfeld's conjecture (\cite{Goldfeld1979}), which states that
% $$N_r(E,X) \sim \frac{1}{2}\sum_{|d| < X}1, \hspace{1cm} X \rightarrow \infty$$
% for $r = 0, 1$. The \emph{weak} Goldfeld conjecture asserts that
% $$N_r(E,X) \gg X, \hspace{1cm} X \rightarrow \infty$$
% for $r = 0, 1$. 

% The main corollary of Theorem \ref{thm:Heegnercorollary} in this note is a verification of this conjecture for a large family of elliptic curves. In particular, we generalize Corollary 18 of loc. cit.

% \begin{theorem}\label{weakGoldfeld}Suppose $E/\mathbb{Q}$ is an elliptic curve such that $E[3]$ is a reducible $\mathrm{Gal}(\overline{\mathbb{Q}}/\mathbb{Q})$-representation. Then a positive proportion of quadratic twists have algebraic and analytic rank equal to $r$, for $r = 0,1$. In particular, the weak Goldfeld conjecture holds for $E$.
% \end{theorem}

The goal in this section is to prove Theorem \ref{thm:3isogeny}. We will need some Davenport-Heilbronn type class number divisibility results due to Nakagawa--Horie and Taya.  For any $x \ge 0$, let $K^+(x)$ denote the set of real quadratic fields $k$ with fundamental discriminant $d_k<x$ and $K^-(x)$ the set of imaginary quadratic fields $k$ with fundamental discriminant $|d_k| < x$. Let $m$ and $M$ be positive integers, and let
\begin{align*}&K^+(x,m,M) := \{k \in K^+(x) : d_k \equiv m \pmod{M}\},\\
&K^-(x,m,M) := \{k \in K^-(x) : d_k \equiv m \pmod{M}\}.
\end{align*}
Recall that we let $h_3(d)$ denote the 3-primary part of the class number of $\mathbb{Q}(\sqrt{d})$, and let $\Phi : \mathbb{Z}_{>0} \rightarrow \mathbb{Z}_{>0}$ denote the Euler totient function. We introduce the following terminology for convenience.
\begin{definition}\label{valid}We say that positive integers $m$ and $M$ comprise a \emph{valid pair $(m,M)$} if both of the following properties hold:
\begin{enumerate}
\item if $\ell$ is an odd prime number dividing $(m,M)$, then $\ell^2$ divides $M$ but not $m$, and 
\item if $M$ is even, then
\begin{enumerate}
\item $4|M$ and $m \equiv 1 \pmod{4}$, or
\item $16|M$ and $m \equiv 8 \; \text{or}\; 12 \pmod{16}$.
\end{enumerate}
\end{enumerate}
\end{definition}
 
Horie and Nakagawa proved the following.

\begin{theorem}[\cite{Nakagawa1988}]\label{HorieNakagawa}
We have
$$|K^+(x,m,M)| \sim |K^-(x,m,M)| \sim \frac{3x}{\pi^2\Phi(M)}\prod_{\ell|M}\frac{q}{\ell+1} \hspace{.5cm} (x \rightarrow \infty).$$
Suppose furthermore that $(m,M)$ is a valid pair. Then 
\begin{align*}&\sum_{k \in K^+(x,m,M)}h_3(d_k) \sim \frac{4}{3}|K^+(x,m,M)| \hspace{.5cm} (x \rightarrow \infty),\\
&\sum_{k \in K^-(x,m,M)}h_3(d_k) \sim 2|K^-(x,m,M)| \hspace{.5cm} (x \rightarrow \infty).
\end{align*}
Here $f(x) \sim g(x) \hspace{.5cm} (x\rightarrow \infty)$ means that $\lim_{x\rightarrow \infty} \frac{f(x)}{g(x)} = 1$, $\ell$ ranges over primes dividing $M$, $q = 4$ if $\ell = 2$, and $q = \ell$ otherwise.
\end{theorem}

Now put
\begin{align*}&K_*^+(x,m,M) := \{k \in K^+(x,m,M) : h_3(d_k) = 1\},\\
&K_*^-(x,m,M) := \{k \in K^-(x,m,M) : h_3(d_k) = 1\}.
\end{align*}

Taya \cite{Taya2000} proves the following bound using Theorem \ref{HorieNakagawa}.
\begin{proposition}\label{positiveproportion}Suppose $(m,M)$ is a valid pair. Then
\begin{align*}&\lim_{x\rightarrow \infty}\frac{|K_*^+(x,m,M)|}{|K^+(x,1,1)|} \ge \frac{5}{6\Phi(M)}\prod_{\ell|M}\frac{q}{\ell+1},\\
&\lim_{x\rightarrow \infty}\frac{|K_*^-(x,m,M)|}{|K^-(x,1,1)|} \ge \frac{1}{2\Phi(M)}\prod_{\ell|M}\frac{q}{\ell+1}.
\end{align*}
In particular, the of real (resp. imaginary) quadratic fields $k$ such that $d_k \equiv m \pmod{M}$ and $h_3(d_k) = 1$ has positive density in the set of all real (resp. imaginary) quadratic fields. 
\end{proposition}
\begin{proof}This follows from the trivial bounds $K_*^+(x,m,M) + 3(K^+(x,m,M) - K_*^+(x,m,M)) \le \sum_{k \in K^+(x,m,M)}h_3(d_k)$ and $K_*^-(x,m,M) + 3(K^-(x,m,M) - K_*^-(x,m,M)) \le \sum_{k \in K^+(x,m,M)}h_3(d_k)$, and the asymptotic formulas from Theorem \ref{HorieNakagawa}.
\end{proof}

We have the following positive density result.

\begin{theorem}\label{twistpositiveproportion} Suppose $A/\mathbb{Q}$ is any $\GL_2$-type abelian variety of conductor $N = N_+N_-N_0$ which has a rational 3-isogeny. Let $d$ be the fundamental discriminant corresponding to the quadratic character $\psi$. Suppose that
\begin{enumerate}
\item $\psi(3) \neq 1$ and $(\psi^{-1}\omega)(3) \neq 1$;
\item $\ell\neq 3, \ell|N_{\splt}$ implies $\psi(\ell) = -1$;
\item $\ell\neq 3, \ell|N_{\nonsplit}$ implies $\psi(\ell) = 1$;
\item $\ell|N_{\add}, \ell \equiv 1 \pmod{3}$ implies $\psi(\ell) = -1$ or $0$;
\item $\ell|N_{\add}, \ell \equiv 2 \pmod{3}$ implies $\psi(\ell) = 0$.
\end{enumerate}
Let 
\begin{equation}\label{Kconditions}d_0 := \begin{cases} d, & d > 0,\\
-3d, & d < 0, d \not\equiv 0 \pmod{3},\\
-d/3, \pmod{M} & d < 0, d \equiv 0 \pmod{3},\\
\end{cases}
\end{equation}
let
$$r(A) := \begin{cases} 1, & 2\nmid \mathrm{lcm}(N,d^2),\\
2, & 2||\mathrm{lcm}(N,d^2),\\
\ord_2(\mathrm{lcm}(N,d^2,16))-1, & 4|\mathrm{lcm}(N,d^2),\\
\end{cases}
$$
and let
$$s_3(d) := \begin{cases}0, & d > 0, d\not\equiv 0 \pmod{3},\; \text{or} \;d < 0, d \equiv 0 \pmod{3},\\
1, & d > 0, d \equiv 0 \pmod{3},\; \text{or}\; d < 0, d\not\equiv 0 \pmod{3}.
\end{cases}
$$
Then a proportion of at least
\begin{equation}\label{proportion'}\frac{d_0}{2^{r(A)+s_3(d)}\cdot 3}\prod_{\ell|N_{\splt}N_{\nonsplit}, \ell\nmid d, \ell \;\text{odd}, \ell\neq 3}\frac{1}{2}\prod_{\ell|N_{\add},\ell\nmid d, \ell \; \text{odd}, \ell \neq 3}\frac{1}{2}\prod_{\ell|d,\ell\;\text{odd}, \ell \neq 3}\frac{1}{2\ell}\prod_{\ell|3N}\frac{q}{\ell+1}
\end{equation}
of all imaginary quadratic fields $K$ have the following properties:
\begin{enumerate}
\item $d_K$ is odd,
\item $K$ satisfies the Heegner hypothesis with respect to $3N$,
\item $h_3(d_0d_K) = 1$.
\end{enumerate}
If furthermore, we impose the assumption on $A$ that
\begin{enumerate}\setcounter{enumi}{5}
\item $h_3(-3d) = 1$ if $\psi(-1) = 1$, and $h_3(d) = 1$ if $\psi(-1) = -1$
\end{enumerate}
then at least the same proportion (\ref{proportion'}) of all imaginary quadratic fields $K$ have:
\begin{enumerate}
\item $d_K$ is odd,
\item $K$ satisfies the Heegner hypothesis with respect to $3N$, and 
\item the Heegner point $P \in A(K)$ is non-torsion. 
\end{enumerate}
\end{theorem}

\begin{proof}
We will apply Proposition \ref{positiveproportion}, as well as Theorem \ref{thm:Heegnercorollary}. Let $N'$ denote the prime-to-3 part of $N$. We first divide into two cases $(a)$ and $(b)$ regarding $d$, corresponding to
\begin{enumerate}[label=$(\alph*)$]
\item $d > 0$ and $d\not\equiv 0 \pmod{3}$, or $d < 0$ and $d \equiv 0 \pmod{3}$;
\item $d > 0$ and $d \equiv 0 \pmod{3}$, or $d < 0$ and $d\not\equiv 0 \pmod{3}$.
\end{enumerate}
We then define a positive integer $M$ as follows:
\begin{enumerate}
\item In case $(a)$, let
$$M = \begin{cases} 3\cdot\mathrm{lcm}(N',d^2,4), & 2\nmid \mathrm{lcm}(N',d^2),\\
3\cdot\mathrm{lcm}(N',d^2,8), & 2||\mathrm{lcm}(N',d^2),\\
3\cdot\mathrm{lcm}(N',d^2,16), & 4|\mathrm{lcm}(N',d^2).\\
\end{cases}$$
\item In case $(b)$, let 
$$M = \begin{cases} 9\cdot\mathrm{lcm}(N',d^2,4), & 2\nmid \mathrm{lcm}(N',d^2),\\
9\cdot\mathrm{lcm}(N',d^2,8), & 2||\mathrm{lcm}(N',d^2),\\
9\cdot\mathrm{lcm}(N',d^2,16), & 4|\mathrm{lcm}(N',d^2).\\
\end{cases}$$
\end{enumerate}

Using the Chinese remainder theorem, choose a positive integer $m$ such that
\begin{enumerate}
\item $m \equiv 2 \pmod{3}$ in case $(a)$, or $m \equiv 3 \pmod{9}$ in case $(b)$,
\item $\ell$ odd prime, $\ell \neq 3$, $\ell|N_{\splt} \implies \frac{m}{d_0} \equiv [\text{quadratic residue unit}] \pmod{\ell}$, and $2|N_{\splt} \implies \frac{m}{d_0} \equiv 1 \pmod{8}$,
\item $\ell$ odd prime, $\ell \neq 3$, $\ell|N_{\nonsplit} \implies \frac{m}{d_0} \equiv [\text{quadratic residue unit}] \pmod{\ell}$, and $2|N_{\nonsplit} \implies \frac{m}{d_0} \equiv 1 \pmod{8}$,
\item $\ell$ prime, $\ell \equiv 1 \pmod{3}$, $\ell|N_{\add}, \ell\nmid d \implies \frac{m}{d_0} \equiv [\text{quadratic residue unit}] \pmod{\ell}$, and $\ell \equiv 1 \pmod{3}, \ell|N_{\add} \implies m \equiv \implies \frac{m}{d_0} \equiv [\text{quadratic residue unit}] \pmod{\ell}$,
\item $\ell$ prime, $\ell$ odd, $\ell \equiv 2 \pmod{3}$, $\ell|N_{\add}$ (which by our assumptions implies $\ell|d$) $\implies m \equiv 0 \pmod{\ell}$ where $\frac{m}{d_0} \equiv [\text{quadratic residue unit}] \pmod{\ell}$, and $2|N_{\add} \implies m \equiv d \pmod{16},$
\end{enumerate}
and furthermore, if $2\nmid N$, then suppose $m \equiv d \pmod{4}$.

%Suppose $d > 0$ (resp. $d < 0$). Suppose $K$ is any imaginary quadratic field such that $dd_K \equiv m \pmod{M}$ where $m \equiv 2 \pmod{3}$ if $d \not\equiv 0 \pmod{3}$ and $m \equiv 3 \pmod{9}$ if $d \equiv 0 \pmod{3}$ (resp. $-3dd_K \equiv m \pmod{M}$ where $m \equiv 3 \pmod{9}$ if $d \not\equiv 0 \pmod{3}$, and $-dd_K/3 \equiv m \pmod{M}$ where $m \equiv 2 \pmod{3}$ if $d \equiv 0 \pmod{3}$). 
%Since $d_K$ is odd, we must have $d_K \equiv 1 \pmod{4}$, and this is compatible with condition (6) which forces $d_K \equiv 1 \pmod{8}$, which in turn forces 2 to split in $K$. 

Suppose $K$ is any imaginary quadratic field such that $d_0d_K \equiv m \pmod{M}$.
Then the congruence conditions corresponding to (1)-(5) above, along with assumptions (1)-(5) in the statement of the theorem, imply
\begin{enumerate}
\item 3 splits in $K$,
\item $\ell \neq 3, \ell|N_{\splt} \implies \ell$ splits in $K$, 
\item $\ell \neq 3, \ell|N_{\nonsplit} \implies \ell$ splits in $K$, 
\item $\ell$ prime, $\ell \equiv 1 \pmod{3}, \ell|N_{\add} \implies \ell$ splits in $K$,
\item$\ell$ prime, $\ell \equiv 2 \pmod{3}, \ell|N_{\add}\implies \ell$ splits in $K$,
\end{enumerate}
and $d_K \equiv 1 \pmod{4}$ (i.e. $d_K$ is odd). Hence $K$ satisfies the Heegner hypothesis with respect to $3N$.

Moreover, the congruence conditions above imply that $(m,M)$ is a valid pair (see Definition \ref{valid}), and the assumptions (4)-(5) in the statement of the theorem imply that $(jd,d^2)$ is also a valid pair whenever $(j,d) = 1$. Thus, by Proposition \ref{positiveproportion}, for any $d_0|M$,
\begin{equation}\label{inequality}\lim_{x\rightarrow \infty}\frac{|K_*^-(x,m,M)|}{|K^-(x/d_0,1,1)|} \ge \frac{d_0}{2\Phi(M)}\prod_{\ell|M}\frac{q}{\ell+1}.
\end{equation}
The left-hand side of (\ref{inequality}) is the proportion of imaginary quadratic $K$ satisfying $d_0d_K \equiv m \pmod{M}$ and $h_3(d_0d_K) = 1$. Moreover, notice that there are 
\begin{align*}\prod_{\ell|N_{\splt}N_{\nonsplit}, \ell\nmid d, \ell \;\text{odd}, \ell\neq 3}\frac{\ell-1}{2}\prod_{\ell|N_{\add},\ell\nmid d, \ell \; \text{odd}, \ell\neq 3}\frac{\ell(\ell-1)}{2}\prod_{\ell|d,\ell\;\text{odd},\ell\neq 3}\frac{\ell-1}{2}
\end{align*}
choices for residue classes of $m$ mod $M$. Combining all the above and summing over each valid residue class $m$ mod $M$, we immediately obtain our lower bound (\ref{proportion1}) for the proportion of imaginary quadratic fields $K$ such that (1) $d_K$ is odd, (2) $K$ satisfies the Heegner hypothesis with respect to $3N$, and (3) $h_3(d_0d_K) = 1$. This proves the part of the theorem before assumption (6) is introduced in the statement.

%Thus for imaginary quadratic $K$ such that $dd_K \equiv m \pmod{M}$ where $m \equiv 2 \pmod{3}$ if $d \not\equiv 0 \pmod{3}$ and $m\equiv 3 \pmod{9}$ if $d \equiv 0 \pmod{3}$ (resp. $-3dd_K \equiv m \pmod{M}$ where $m \equiv 3 \pmod{9}$ if $d \not\equiv 0 \pmod{3}$, and $-dd_K/3 \equiv m \pmod{M}$ where $m \equiv 2 \pmod{3}$ if $d \equiv 0 \pmod{3}$), $(E,3,\psi,K)$ satisfies all the congruence conditions of Theorem \ref{thm:Heegnercorollary} except for possibly $h_3(dd_K) = 1$ (resp. $h_3(-3dd_K) = 1$). 

If we assume that $A$ satisfies assumption (6) in the statement of the theorem, then for all $K$ as above, we see that $A$, $p = 3$ and $K$ satisfy all the assumptions of Theorem \ref{thm:Heegnercorollary} (see Remarks \ref{remark:3isogeny} and \ref{p=3remark}), thus implying that $P$ is non-torsion. The final part of the theorem now follows.

\end{proof}

Similarly, we have the following positive density result for producing $A$ which satisfy the assumptions of Theorem \ref{twistpositiveproportion}.

\begin{theorem}\label{realtwistpositiveproportion} Suppose $(N_1,N_2,N_3)$ is a triple of pairwise coprime integers such that $N_1N_2$ is square-free, $N_3$ is square-full and $N_1N_2N_3 = N$. Let
$$r := \begin{cases} 0, & 2\nmid N,\\
2, & 2|N.\\
\end{cases}$$
Then a proportion of at least
\begin{align*}&\frac{1}{2^r\cdot 3}\prod_{\ell|N_1N_2,\ell \; \text{odd}, \ell\neq 3}\frac{1}{2}\prod_{\ell|N_3,\ell \; \text{odd}, \ell \neq 3}\frac{1}{\ell}\prod_{\ell|N, \ell\neq 3}\frac{q}{\ell+1}
\end{align*}
of even (resp. odd) quadratic characters $\psi$ corresponding to real (resp. imaginary) quadratic fields $\mathbb{Q}(\sqrt{d})$, where the $d > 0$ (resp. $d < 0$) are fundamental discriminants, satisfy
\begin{enumerate}
\item $\psi(3) \neq 1$ and $(\psi^{-1}\omega)(3) \neq 1$;
\item $\ell\neq 3, \ell|N_1$ implies $\psi(\ell) = -1$;
\item $\ell\neq 3, \ell|N_2$ implies $\psi(\ell) = 1$;
\item $\ell\neq 3, \ell|N_3, \ell \equiv 1 \pmod{3}$ implies $\psi(\ell) = 0$; %or -1;
\item $\ell\neq 3, \ell|N_3, \ell \equiv 2 \pmod{3}$ implies $\psi(\ell) = 0$;
\item $h_3(-3d) = 1$ (resp. $h_3(d) = 1$).
\end{enumerate}
Moreover, we have that for any $i \in \{2, 3, 5, 8\}$,
\begin{itemize}
\item $1/4$ of the above fundamental discriminants $d > 0$ (resp. $d < 0$) satisfy $d \equiv i \pmod{9}$.
\end{itemize}
\end{theorem}

\begin{proof}
We will apply Proposition \ref{positiveproportion}. Using the Chinese remainder theorem, choose a positive integer $m$ which satisfies the following congruence conditions:
\begin{enumerate}
\item $m \equiv 3 \pmod{9}$ or $m \equiv 2 \pmod{3}$,
\item $\ell$ odd prime, $\ell \neq 3$, $\ell|N_1 \implies m \equiv -3[\text{quadratic non-residue}] \pmod{\ell}$, and $2|N_1 \implies m \equiv 1 \pmod{8}$,
\item $\ell$ odd prime, $\ell \neq 3$, $\ell|N_2 \implies m \equiv -3[\text{quadratic residue unit}] \pmod{\ell}$, and $2|N_2 \implies m \equiv 5 \pmod{8}$,
\item $\ell$ odd prime, $\ell \neq 3$, $\ell|N_3, \ell\equiv 1 \pmod{3} \implies m \equiv %-3[\text{quadratic non-residue}] \pmod{\ell}$ or $m \equiv
0 \pmod{\ell}$ and $m \not\equiv 0 \pmod{\ell^2}$,
\item $\ell$ odd prime, $\ell \neq 3$, $\ell|N_3, \ell\equiv 2 \pmod{3} \implies m \equiv 0 \pmod{\ell}$ and $m\not\equiv 0  \pmod{\ell^2}$, and $2|N_3 \implies m \equiv 8$ or $12 \pmod{16}$. 
\end{enumerate}

%Here (and throughout the proof) again, conditions (1)-(7) correspond to the desired conditions (1)-(7) on $\psi$ in the statement of the theorem, and (8) is an extra condition required to apply the theorem of Horie-Nakagawa. Let $N'$ denote the prime-to-3 part of $N$. 

Let $N'$ denote the prime-to-3 part of $N$. Given such an $m$, let a positive integer $M$ be defined as follows:
\label{M}\begin{itemize}
\item If $m \equiv 3 \pmod{9}$, let 
$$M = \begin{cases}9N', & 2\nmid N,\\
9\cdot\mathrm{lcm}(N',8), & 2||N,\\
9\cdot\mathrm{lcm}(N',16), &4|N.\\
\end{cases}$$
\item If $m \equiv 2 \pmod{3}$, let 
$$M =\begin{cases} 3N', & 2\nmid N,\\
3\cdot\mathrm{lcm}(N',8), & 2||N,\\
3\cdot\mathrm{lcm}(N',16), & 4|N.\\
\end{cases}$$
\end{itemize}

If $m \equiv 2 \pmod{3}$, suppose $d$ is a fundamental discriminant with
\begin{itemize}
\item $d > 0, d \equiv 0 \pmod{3}$, and $-d/3 \equiv m \pmod{M}$, or
\item $d < 0, d \not\equiv 0 \pmod{3}$, and $d \equiv m \pmod{M}$.
\end{itemize}

If $m \equiv 3 \pmod{9}$, suppose $d$ is a fundamental discriminant with
\begin{itemize}
\item $d > 0, d \not\equiv 0 \pmod{3}$, and $-3d \equiv m \pmod{M}$, or
\item $d < 0, d \equiv 0 \pmod{3}$, and $d \equiv m \pmod{M}$.
\end{itemize}

%Suppose $d > 0$ (resp. $d < 0$) is a fundamental discriminant with $d \not\equiv 0 \pmod{3}$ (resp. $d\not\equiv 0 \pmod{3}$), and $-3d \equiv m \pmod{M}$ (resp. $d \equiv m \pmod{M}$) where $m \equiv 3 \pmod{9}$ (resp. $m \equiv 2 \pmod{3}$), or $d \equiv 0 \pmod{3}$ (resp. $d\equiv 0 \pmod{3}$) and $-d/3 \equiv m \pmod{M}$ (resp $d \equiv m \pmod{M}$) if $m \equiv 2 \pmod{3}$ (resp. $m \equiv 3 \pmod{9}$).

Let $\psi$ be the quadratic character associated with $d$. Then the congruence conditions on $m$ corresponding to (1)-(5) above imply
\begin{enumerate}
\item $\psi(3) \neq 1$ and $(\psi^{-1}\omega)(3) \neq 1$;
\item $\ell \neq 3$ prime, $\ell|N_1 \implies \psi(\ell) = -1$;
\item $\ell \neq 3$ prime, $\ell|N_2 \implies \psi(\ell) = 1$;
\item $\ell \neq 3$ prime, $\ell|N_3, \ell\equiv 1 \pmod{3} \implies \psi(\ell) = %-1$ or $
0$;
\item $\ell \neq 3$ prime, $\ell|N_3, \ell\equiv 2 \pmod{3} \implies \psi(\ell) = 0$.
\end{enumerate}
Thus $\psi$ satisfies the desired congruence conditions (1)-(5) in the statement of the theorem. Now we address (6). The congruence conditions (1)-(5) above imply that $(m,M)$ is a valid pair. Thus, by Proposition \ref{positiveproportion}, if $m \equiv 2 \pmod{3}$ with corresponding $M$ as defined above, then
\begin{equation}\label{inequality1}\lim_{x\rightarrow \infty}\frac{|K_*^-(x,m,M)|}{|K^+(3x,3,9)| + |K^+(3x,6,9)|}\ge \frac{1}{6\Phi(M)}\prod_{\ell|M,\ell\neq 3}\frac{q}{\ell+1}
\end{equation}
where the left-hand side of (\ref{inequality1}) is the proportion of $d > 0$ which satisfy $d \equiv 0 \pmod{3}$ and $-d/3\equiv m \pmod{M}$ and $h_3(-3d)  = h_3(-d/3) = 1$, and 
\begin{equation}\label{inequality2}\lim_{x\rightarrow \infty}\frac{|K_*^-(x,m,M)|}{|K^-(x,1,3)| + |K^-(x,2,3)|}\ge \frac{1}{2\Phi(M)}\prod_{\ell|M,\ell\neq 3}\frac{q}{\ell+1}
\end{equation}
where the left-hand side of (\ref{inequality2}) is the proportion of $d < 0$ which satisfy $d \not\equiv 0 \pmod{3}$, $d \equiv m \pmod{M}$ and $h_3(d) = 1$. Similarly by Proposition \ref{positiveproportion}, if $m \equiv 3 \pmod{9}$ with corresponding $M$ as defined above, then
\begin{equation}\label{inequality3}\lim_{x\rightarrow \infty}\frac{|K_*^-(x,m,M)|}{|K^+(x/3,1,3)| + |K^+(x/3,2,3)|}\ge \frac{3}{2\Phi(M)}\prod_{\ell|M, \ell\neq 3}\frac{q}{\ell+1}
\end{equation}
where the left-hand side of (\ref{inequality3}) is the proportion of $d > 0$ which satisfy $d \not\equiv 0 \pmod{3}$, $-3d \equiv m \pmod{M}$ and $h_3(-3d) = 1$, and
\begin{equation}\label{inequality4}\lim_{x\rightarrow \infty}\frac{|K_*^-(x,m,M)|}{|K^-(x,1,3)| + |K^-(x,2,3)|}\ge \frac{1}{2\Phi(M)}\prod_{\ell|M, \ell\neq 3}\frac{q}{\ell+1}
\end{equation}
where the left-hand side of (\ref{inequality4}) is the proportion of $d < 0$ which satisfy $d \equiv 0 \pmod{3}$, $d \equiv m \pmod{M}$ and $h_3(d) = 1$.

Moreover, in each case, we have 
\begin{align*}&\prod_{\ell|N_1,\ell \; \text{odd}, \ell \neq 3}\frac{\ell-1}{2}\prod_{\ell|N_2,\ell \; \text{odd}, \ell \neq 3}\frac{\ell-1}{2}\\
&\cdot\prod_{\ell|N_3,\ell \; \text{odd}, \ell \equiv 1 \pmod{3}}%\frac{(\ell+2)(\ell-1)}{2}
(\ell-1)\prod_{\ell|N_3,\ell \; \text{odd}, \ell \equiv 2 \pmod{3}}(\ell-1)\prod_{\text{if} \; 2|N_3}2
\end{align*}
choices of residue classes $m$ mod $M$ which satisfy congruence conditions (1)-(5). Combining all the above and summing over each these residue class $m$ mod $M$, we immediately obtain our lower bounds for the proportions of desired $d > 0$ from (\ref{inequality2}) and desired $d < 0$ from (\ref{inequality3}).

The final part of the theorem follows by directly counting the number of residue classes $m$ mod $M$ which force $d \equiv i \pmod{9}$ for $i \in \{2,3,5,8\}$.
\end{proof}

\begin{remark}\label{atleastoneremark} Let $\lambda$ be the prime above $p = 3$ fixed in the beginning of \S \ref{sec:mainthm}. Suppose for $A$ as above, $A[\lambda]^{\mathrm{ss}} \cong \mathbb{F}_{\lambda}\oplus \mathbb{F}_{\lambda}(\omega)$. For a fundamental discriminant $d$, let $A^{(d)}$ denote the quadratic twist of $A$ by $d$. Note that for each $d$ produced by Theorem \ref{realtwistpositiveproportion}, Theorem \ref{twistpositiveproportion} shows that there is a positive proportion of imaginary quadratic $K$ satisfying the Heegner hypothesis with respect to $Nd^2$ such that the corresponding Heegner point $P\in A^{(d)}(K)$ is non-torsion. In particular, for each such $d$ there is \emph{at least one} $K$ such that $P \in A^{(d)}(K)$ is non-torsion. Thus $r_{\mathrm{an}}(A^{(d)})= \dim A\cdot\frac{1-w(A^{(d)})}{2}$.
\end{remark}

\begin{theorem}\label{thm:av3iso}
    The weak Goldfeld Conjecture is true for any abelian variety $A/\mathbb{Q}$ of $\GL_2$-type with a rational 3-isogeny. Namely, there is a positive proportion of quadratic twists of $A/\mathbb{Q}$ of analytic rank 0 (resp. analytic rank equal to $\dim A$).
\end{theorem}

\begin{proof}
Suppose $A$ has a 3-isogeny defined over $\mathbb{Q}$ and that (without loss of generality) $A$ satisfies our assumptions stated before Theorem \ref{thm:Heegnercorollary}, i.e. $A[\lambda]^{\mathrm{ss}} \cong \mathbb{F}_{\lambda}(\psi)\oplus\mathbb{F}_{\lambda}(\psi^{-1}\omega)$ for some quadratic character $\psi : \mathrm{Gal}(\overline{\mathbb{Q}}/\mathbb{Q}) \rightarrow \mu_2$. Twisting by the quadratic character $\psi^{-1}$, we may assume without loss of generality that $A[\lambda]^{\mathrm{ss}} \cong \mathbb{F}_{\lambda}\oplus\mathbb{F}_{\lambda}(\omega)$.

%We now apply Theorems \ref{realtwistpositiveproportion} and \ref{twistpositiveproportion} to produce positive proportions of quadratic twists $d$ and imaginary quadratic fields $K$ satisfying the Heegner hypothesis with respect to $3Nd^2$ such that $r_\mathrm{an}(E^{(d)}/K) = 1$. It will follow that one of $r_\mathrm{an}(E^{(d)})$  and $r_\mathrm{an}(E^{(dd_K)})$ is 0 and the other is 1, according to their root numbers. For the rest of the proof, given a fundamental discriminant $d$, let $\psi$ denote the quadratic character associated with $\mathbb{Q}(\sqrt{d})$. Recall that viewing $\psi$ as an id\'{e}lic Hecke character, we let $\psi_{\ell}$ denote the local character at $\ell$.

Let $d$ be a fundamental discriminant corresponding to a quadratic character $\psi$ in the family of $d$ produced by Theorem \ref{realtwistpositiveproportion} (with the integers $N_1 = N_+, N_2 = N_-$ and $N_3 = N_0$ as in our setting). In particular, $A^{(d)}[\lambda]^{\mathrm{ss}} \cong \mathbb{F}_{\lambda}(\psi)\oplus \mathbb{F}_{\lambda}(\psi^{-1}\omega)$ satisfies the assumptions of Theorem \ref{twistpositiveproportion}, including assumption (6). Hence, we can apply Theorem \ref{twistpositiveproportion} to $A^{(d)}$ to conclude that a positive proportion of imaginary quadratic fields $K$ satisfy the Heegner hypothesis with respect to $3Nd^2$ and have that the associated Heegner point $P \in A^{(d)}(K)$ is non-torsion. Since $w(A^{(d)})w(A^{(dd_K)}) = w(A/K) = -1$ (the last equality following from the Heegner hypothesis), we have that each such $K$ satisfies
$$r_{\mathrm{an}}(A^{(dd_K)}) = \dim A\cdot \frac{1 + w(A^{(d)})}{2}.$$
Hence there are a positive proportion of quadratic twists of $A$ with rank $\dim A\cdot\frac{1+w(A^{(d)})}{2}$, and in fact by Theorem \ref{twistpositiveproportion}, a lower bound for this proportion is given by
\begin{equation}\label{proportion1}\frac{d_0}{2^{r(A^{(d)})+s_3(d)}\cdot 3}\prod_{\ell|N_{\splt}N_{\nonsplit},\atop \ell\nmid d,\ell \;\text{odd}, \ell \neq 3}\frac{1}{2}\prod_{\ell|N_{\add}d^2,\atop \ell\nmid d, \ell \;\text{odd}, \ell\neq 3}\frac{1}{2}\prod_{\ell|d, \;\text{odd}, \ell\neq 3}\frac{1}{2\ell}\prod_{\ell|3Nd^2}\frac{q}{\ell+1}
\end{equation}
in the notation of the statement of the theorem.

Now choose any $K$ as produced by Theorem \ref{twistpositiveproportion} for $A^{(d)}$, so that $w(A^{(dd_K)}) = -w(A^{(d)})$. In particular, $d_K$ is odd and prime to $3Nd$. Then by construction $h_3(dd_K) = 1$ if $d > 0$ and $h_3(-3dd_K) = 1$ if $d < 0$, and so $A^{(dd_K)}[\lambda]^{\mathrm{ss}} \cong \mathbb{F}_{\lambda}(\psi\varepsilon_K)\oplus \mathbb{F}_{\lambda}((\psi\varepsilon_K)^{-1}\omega)$ satisfies all of the assumptions (including (6)) of Theorem \ref{twistpositiveproportion}. Hence, we can apply Theorem \ref{twistpositiveproportion} to $A^{(dd_K)}$ to conclude that a positive proportion of imaginary quadratic fields $K'$ satisfy the Heegner hypothesis with respect to $3Nd^2d_K^2$ and have that the associated Heegner point $P \in A^{(dd_K)}(K')$ is non-torsion. Since $w(A^{(dd_K)})w(A^{(dd_Kd_{K'})}) = w(A^{(dd_K)}/K') = -1$, we have that each such $K'$ satisfies
\begin{equation}\label{proportion2}r_{\mathrm{an}}(A^{(dd_Kd_{K'})}) = \dim A\cdot\frac{1 + w(A^{(dd_K)})}{2} = \dim A\cdot\frac{1 - w(A^{(d)})}{2}.
\end{equation}
Hence there are a positive proportion of quadratic twists of $A$ with rank $\dim A\cdot\frac{1-w(A^{(d)})}{2}$, and in fact by Theorem \ref{twistpositiveproportion}, a lower bound for this proportion is given by
$$\frac{(dd_K)_0}{2^{r(A^{(dd_K)})+s_3(dd_K)}\cdot 3}\prod_{\ell|N_{\splt}N_{\nonsplit},\atop \ell\nmid dd_K,\ell \;\text{odd}, \ell \neq 3}\frac{1}{2}\prod_{\ell|N_{\add}(dd_K)^2,\atop \ell\nmid dd_K, \ell \;\text{odd}, \ell\neq 3}\frac{1}{2}\prod_{\ell|dd_K, \;\text{odd}, \ell\neq 3}\frac{1}{2\ell}\prod_{\ell|3N(dd_K)^2}\frac{q}{\ell+1}$$
in the notation of the statement of the theorem. (Note that in fact $r(A^{(dd_K)}) = r(A^{(d)})$ since $d_K$ is odd.)
% We have thus established Theorem \ref{thm:3isogeny}.
\end{proof}

Suppose now $A = E$ is an elliptic curve over $\mathbb{Q}$, and so $\lambda = 3$. When $E$ is semistable, we have $E[3]^{\mathrm{ss}} \cong \mathbb{F}_3\oplus \mathbb{F}_3(\omega)$ for the following reason: Suppose $E[3]^{\mathrm{ss}} \cong \mathbb{F}_3(\psi) \oplus \mathbb{F}_3(\psi^{-1}\omega)$ for some quadratic character $\psi$. Then $\psi$ cannot be ramified at any $\ell||N$ since the corresponding admissible $\GL_2(\mathbb{Q}_{\ell})$ representation is Steinberg of conductor $\ell$, but if $\psi$ was ramified at $\ell$ it would force the conductor to be divisible by $\ell^2$ by the above description of $E[3]^{\mathrm{ss}}$. Hence $\psi$ is a quadratic character only possibly ramified at $3$ and hence must be either 1 or $\omega$. 

Now we can use Theorem \ref{realtwistpositiveproportion} to compute explicit lower bounds on the proportion of rank 0 and rank 1 quadratic twists.

\begin{proposition}\label{semistableprop}Let $E/\mathbb{Q}$ be semistable and suppose that $E$ has a rational 3-isogeny. 

If $3\nmid N$, then in the notation of Theorem \ref{realtwistpositiveproportion} (with $N_1 = N_{\splt}, N_2 = N_{\nonsplit}$, and $N_3 = N_{\add} = 1$, at least 
\begin{equation}\frac{1}{2^r\cdot 3}\prod_{\ell|N,\ell \; \text{odd}, \ell\neq 3}\frac{1}{2}\prod_{\ell|N, \ell\neq 3}\frac{q}{\ell+1}
\end{equation}
of $d > 0$ (resp. $d < 0$) have $r_{\mathrm{an}}(E^{(d)}) = 1$ (resp. $r_{\mathrm{an}}(E^{(d)}) = 0$). 

If $3|N$, then:
\begin{enumerate}
\item If 3 is of split multiplicative reduction, then at least 
\begin{equation}\frac{1}{2^r\cdot 3}\prod_{\ell|N,\ell \; \text{odd}, \ell\neq 3}\frac{1}{2}\prod_{\ell|N, \ell\neq 3}\frac{q}{\ell+1}
\end{equation}
of $d > 0$ (resp. $d < 0$) have $r_{\mathrm{an}}(E^{(d)}) = 1$ (resp. $r_{\mathrm{an}}(E^{(d)}) = 0$).
\item If 3 is of nonsplit multiplicative reduction, then at least 
\begin{equation}\frac{1}{2^{r+2}\cdot 3}\prod_{\ell|N,\ell \; \text{odd}, \ell\neq 3}\frac{1}{2}\prod_{\ell|N, \ell\neq 3}\frac{q}{\ell+1}
\end{equation}
of $d > 0$ (resp. $d < 0$) have $r_{\mathrm{an}}(E^{(d)}) = 0$ (resp.  $r_{\mathrm{an}}(E^{(d)}) = 1$), and at least
\begin{equation}\frac{1}{2^{r+2}}\prod_{\ell|N,\ell \; \text{odd}, \ell\neq 3}\frac{1}{2}\prod_{\ell|N, \ell\neq 3}\frac{q}{\ell+1}
\end{equation}
of $d > 0$ (resp. $d < 0$) have $r_{\mathrm{an}}(E^{(d)}) = 1$ (resp. $r_{\mathrm{an}}(E^{(d)}) = 0$).
\end{enumerate}
\end{proposition}
\begin{proof}First we apply Theorem \ref{realtwistpositiveproportion} to $N_1 = N_{\splt}, N_2 = N_{\nonsplit}$, and $N_3 = N_{\add} = 1$. For any $d$ produced by the theorem, Remark \ref{atleastoneremark} implies that 
\begin{equation}\label{rankformula}r_{\mathrm{an}}(E^{(d)}) = \frac{1-w(E^{(d)})}{2}.
\end{equation}

Let $d$ be any fundamental discriminant produced by Theorem \ref{realtwistpositiveproportion}. By the properties of the $d$ produced in Theorem \ref{realtwistpositiveproportion}, the corresponding local characters $\psi_{\ell}$ for satisfy the implications
\begin{equation}\label{imply2}\ell|N, \ell\nmid d \implies \ell||N \implies \psi_{\ell}(\ell) w_{\ell}(E) = -\psi_{\ell}(\ell)a_{\ell}(E) = -\psi(\ell)a_{\ell}(E) = 1
\end{equation}
(where the last chain of equalities follows since for $\ell||N$, $w_{\ell}(E) = -a_{\ell}(E)$), and furthermore since $N = N_{\splt}N_{\nonsplit}$ (since we assume that $E$ is semistable),
\begin{equation}\label{imply3}\ell|(N,d) \implies \ell = 3.
\end{equation}

We now calculate $w(E^{(d)})$ using (\ref{imply2}) and (\ref{imply3}). Since $E$ is semistable, the global root number $w(E^{(d)})$ is computed via changes to local root numbers $w_{\ell}(E)$ under the quadratic twist by $d$ as follows (see \cite[Table 1]{Balsam2014}):
\begin{enumerate}
\item if $\ell\nmid Nd$, then $w_{\ell}(E^{(d)}) = w_{\ell}(E) = 1$;
\item if $\ell|N, \ell\nmid d$, then $w_{\ell}(E^{(d)}) = \psi_{\ell}(\ell)w_{\ell}(E) = 1$;
\item if $\ell\nmid N, \ell|d$ then $w_{\ell}(E^{(d)}) = \psi_{\ell}(-1)w_{\ell}(E) = \psi_{\ell}(-1)$;
\item if $\ell|(N,d)$, then $\ell = 3$ and $w_3(E^{(d)}) = -\psi_3(-1)w_3(E)$;
\item $w_{\infty}(E^{(d)}) = w_{\infty}(E) = -1$.
\end{enumerate}

Hence
\begin{equation}\label{rootnumber2}w(E^{(d)}) = -\psi(-1)\left(\prod_{\text{if}\; 3|(N,d)}-w_{3}(E)\right).
\end{equation}

If $3\nmid N$, then we have $3\nmid (N,d)$, and so $w(E^{(d)}) = -\psi(-1)$. Thus, by (\ref{rankformula}) and the lower bound given in the statement of Theorem \ref{realtwistpositiveproportion}, in the notation of the theorem we have that at least
\begin{equation}\label{quadraticproportion}\frac{1}{2^r\cdot 3}\prod_{\ell|N,\ell \; \text{odd}, \ell\neq 3}\frac{1}{2}\prod_{\ell|N, \ell\neq 3}\frac{q}{\ell+1}
\end{equation}
of $d > 0$ have $r_{\mathrm{an}}(E^{(d)})  = 1$, and at least the same proportion of $d < 0$ have $r_{\mathrm{an}}(E^{(d)}) =  0$.

If $3|N$, then 
$$w(E^{(d)}) = \begin{cases} -\psi(-1), & 3\nmid d,\\
-\psi(-1), & 3|d, 3 \;\text{is of split multiplicative reduction (i.e. $w_3(E) = -1$)},\\
\psi(-1), & 3|d, 3 \;\text{is of nonsplit multiplicative reduction (i.e. $w_3(E) = 1$)}.
\end{cases}
$$
The desired bounds in this case follow again from (\ref{rankformula}), the lower bound given in the statement of Theorem \ref{realtwistpositiveproportion} and the final part of that theorem.
%\begin{itemize}
%\item If 3 is of split multiplicative reduction, at least the quantity in (\ref{quadraticproportion}) of $d > 0$ have $r_{\mathrm{an}}(E^{(d)}) = 1$, and at least the same proportion of $d < 0$ have $r_{\mathrm{an}}(E^{(d)}) = 0$.
%\item If 3 is of nonsplit multiplicative reduction, at least 1/4 times the quantity in (\ref{quadraticproportion}) of $d > 0$ have $r_{\mathrm{an}}(E^{(d)}) = 0$ and at least 3/4 times the quantity in (\ref{quadraticproportion}) have $r_{\mathrm{an}}(E^{(d)}) = 1$, and at least 3/4 times the quantity in (\ref{quadraticproportion}) of $d < 0$ have $r_{\mathrm{an}}(E^{(d)}) = r_{\mathrm{an}}(E^{(d)}) = 0$ and 1/4 times the quantity in (\ref{quadraticproportion}) of $d < 0$ have $r_{\mathrm{an}}(E^{(d)}) = r_{\mathrm{an}}(E^{(d)}) = 1$.
%\end{itemize}
\end{proof}

%It follows that $w(E^{(d)})$ is $+1$ (resp. $-1$) for a positive proportion of $d$'s under consideration. Hence either a positive proportion of real quadratic twists of $E$ have algebraic and analytic rank 0 and a positive proportion of imaginary quadratic twists of $E$ have algebraic and analytic rank 1, or a positive proportion of imaginary quadratic twists of $E$ have algebraic and analytic rank 0 and a positive proportion of real quadratic twists of $E$ have algebraic and analytic rank 1. In fact, we have the following explicit lower bounds for these proportions given in the statement of Theorem \ref{realtwistpositiveproportion}.

\begin{remark}It is most likely possible to refine the casework in the proofs of Theorems \ref{realtwistpositiveproportion} and \ref{twistpositiveproportion} in order to achieve better lower bounds of twists with ranks 0 or 1.
\end{remark}

\begin{example}\label{exa:19a1} Consider the elliptic curve $$E=19a1: y^2 + y = x^3 + x^2 - 9x - 15$$ in Cremona's labeling. Then $E(\mathbb{Q})\cong\mathbb{Z}/3 \mathbb{Z}$, so we take $p=3$ and obtain $E[3]^\mathrm{ss}=\mathbb{F}_3 \oplus \mathbb{F}_3(\omega)$. Notice that $N=N_\mathrm{split}=19$ and the root number $w(E)=+1$. Consider the set of fundamental discriminant $d>0$ (resp. $d<0$) such that
  \begin{enumerate}
  \item $\psi_d(3)\ne1$ and $(\psi_d\omega)(3)\ne1$.
  \item $\psi_d(19)=-1$.
  \item $h_3(-3d)=1$ (resp. $h_3(d)=1$).
  \end{enumerate}
  The first few such $d>0$ are $$d=8, 12, 21, 41, 53, 56, 65, 84, 89, 129, 164, 165, 185, 189,\cdots$$ and the first few such $d<0$ are $$d=-4, -7, -24, -28, -43, -55, -63, -115, -123, -159, -163, -168, -172, -175, -187, -195,\cdots$$ Notice that the root number $w(E^{(d)})=\psi_d(-19)=-1$ (resp. $+1$), we know  from Theorem \ref{twistpositiveproportion}  that  $$r_\mathrm{an}(E^{(d)})=
  \begin{cases}
    0, & d<0, \\
    1, & d>0.
  \end{cases}$$ The explicit lower bounds in Proposition \ref{semistableprop} show that at least $\frac{19}{120} = 15.833\%$ of real quadratic twists of $E$ have rank 1, and at least $\frac{19}{120} = 15.833\%$ of imaginary quadratic twists of $E$ have rank 0 (compare the lower bound $\frac{19}{240}=7.917\%$ in \cite[p. 640]{James1998}).

\end{example}

\section{The sextic twists family}
\label{sec:sextic-twists-family}

\subsection{The curves $E_d$}
In this section we consider the elliptic curve of $j$-invariant 0, $$E=27a1=X_0(27): y^2=x^3-432.$$ We remind the reader that $E$ has CM by the ring of integers $\mathbb{Z}[\zeta_3]$ of $\mathbb{Q}(\sqrt{-3})$ and is isomorphic to the Fermat cubic curve $X^3+Y^3=1$ via the transformation $$X=\frac{36-y}{6x},\quad Y=\frac{36+y}{6x}.$$%  The associated modular form has the following well-known eta product $$f(z)=q\prod_{n=1}^\infty (1-q^{3n})^2(1-q^{9n})^2=\eta(3z)^2\eta(9z)^2.$$
\begin{definition}\label{def:sextictwists}
  For $d\in \mathbb{Z}$, we denote $E_d$ the $d$-th sextic twist of  $E$, $$E_d: y^2=x^3-432d.$$ Notice that the $d$-th quadratic twist $E^{(d)}$ of $E$ is given by $$E_{d^3}=E^{(d)}: y^2=x^3-432d^3,$$ and the $d$-th cubic twist of $E$ is given by $$E_{d^2}: y^2=x^3-432d^2.$$

\begin{remark}
   The cubic twist $E_{d^2}$ is isomorphic to the curve $X^3+Y^3=d$ and its rational points provide solutions to the classical \emph{sum of two cubes} problem. These equations have a long history, see \cite[\S 1]{Zagier1987} or \cite[\S 1]{Watkins2007} for an overview. %  (see Sylvester (\cite{Sylvester1879}) and Selmer (\cite{Selmer1951}).    
\end{remark}

\end{definition}

% We first observe that cubic twisting does not change the mod 3 Galois representation.

% \begin{lemma}\label{lem:cubictwist}

% \end{lemma}

% \begin{proof}
%   Since $E_d$
% \end{proof}
% \begin{lemma}\label{lem:mod3S3}
%   We have $\mathbb{Q}(E_d[3])=\mathbb{Q}(\zeta_3, \sqrt[3]{d})$.
% \end{lemma}

% \begin{proof}
%   For any elliptic curve $E$, the 3-torsion field $\mathbb{Q}(E[3])$ always contain the extension $\mathbb{Q}(\zeta_3, \sqrt[3]{\Delta})$ (\cite[p. 305]{Serre1972}), where $\Delta$ is the discriminant of a Weierstrass equation of $E$ . Since $E_d$ has a 3-isogeny, we know that the image of the mod 3 Galois representation is contained in the subgroup of $\mathrm{GL}_2(\mathbb{F}_3)$ generated by $\left(\begin{smallmatrix}1 & 1 \\0 &1\end{smallmatrix}\right), \left(\begin{smallmatrix}2 & 0\\ 0 & 1\end{smallmatrix}\right)$. Therefore $\mathbb{Q}(E[3])=\mathbb{Q}(\zeta_3, \sqrt[3]{\Delta})$. The result then follows from the fact that $E_d$ has a 3-isogeny and $\Delta(E_d)=-(432)^3 d^2$.
% \end{proof}

\begin{lemma}\label{lem:mod3rep}
  We have an isomorphism of $G_\mathbb{Q}$-representations $$E_d[3]^\mathrm{ss}\cong \mathbb{F}_3(\psi_d) \oplus \mathbb{F}_3(\psi_{d}\omega).$$ Here $\psi_d: G_\mathbb{Q}\rightarrow\Aut(\mathbb{F}_3)=\{\pm1\}$ is the quadratic character associated to the extension $\mathbb{Q}(\sqrt{d})/\mathbb{Q}$ and $\omega=\psi_{-3}: G_\mathbb{Q}\rightarrow\Aut(\mathbb{F}_3)=\{\pm1\}$. 
\end{lemma}

\begin{proof}
Notice that under cubic twisting the associated modular forms are congruent mod $(\zeta_3-1)$. Since the Hecke eigenvalues are integers, we know that the associated modular forms are indeed congruent mod 3. Hence cubic twisting does not change the semi-simplification of the mod 3 Galois representations.   Notice that $E_d\cong E_{d^7}$ is the $d^4$-th sextic twist of the curve $E_{d^3}$, which is the same as the $d^2$-cubic twist of the quadratic twist $E^{(d)}$. Since $E(\mathbb{Q})[3]\cong \mathbb{Z}/3 \mathbb{Z}$, we have an exact sequence of $G_\mathbb{Q}$-modules, $$0\rightarrow \mathbb{F}_3\rightarrow E[3]\rightarrow \mathbb{F}_3(\omega)\rightarrow0.$$ Hence we have an exact sequence of $G_\mathbb{Q}$-modules $$0\rightarrow \mathbb{F}_3(\psi_d)\rightarrow E^{(d)}[3]\rightarrow \mathbb{F}_3(\psi_d\omega)\rightarrow 0.$$ The result then follows.
\end{proof}

\begin{lemma}\label{lem:rootnumberEd}
  Assume that:
  \begin{enumerate}
  \item $d$ is a fundamental discriminant.
  \item $d\equiv0\pmod{3}$.
  \end{enumerate}
Then the root number of $E_d$ is given by $$w(E_d)=
\begin{cases}
  -\sign(d),& d\equiv3\pmod{9},\\
  \sign(d),& d\equiv6\pmod{9}.
\end{cases}$$
\end{lemma}

\begin{proof}   We use the closed formula for the local root numbers $w_\ell(E_d)$ in \cite[\S 9]{Liverance1995}.
  \begin{enumerate}
  \item Since $d$ is a fundamental discriminant, we have either $d\equiv1\pmod{4}$, or $d=4d'$ for some $d'\equiv3\pmod{4}$, or $d=8d'$ for some $d'\equiv1\pmod{4}$. In the first case we have $-432d=2^4\cdot (-27d)$, with $2\nmid(-27d)$. In the second case we have $-432d=2^6\cdot(-27d')$, and  in the third case we have $-432d=2^7\cdot(-27d')$, with $2\nmid (-27d')$. The local root number formula gives
    \begin{equation}\label{eq:rootnumberat2}
      w_2(E_d)=
    \begin{cases}
      +1, & 2\nmid d\text{ or } 4|| d, \\
      -1, & 8||d.
    \end{cases}
    \end{equation}
  \item Let $d=3d'$. Then $-432d=3^4\cdot (-16d')$, with $3\nmid-16d'$. Since the exponent of 3 is 4, which is $\equiv1\pmod{3}$, we know that $w_3(E_d)=+1$.
  \item Notice that if $2\nmid d$ or $4||d$, then the number of prime factors $\ell|d$ such that $\ell\ge5$ and $\ell\equiv2\pmod{3}$ is odd if and only if $|d'|\equiv2\pmod{3}$. Similarly, if $8||d$, then the number of prime factors $\ell|d$ such that $\ell\ge5$ and $\ell\equiv2\pmod{3}$ is odd if and only if $|d'|\equiv1\pmod{3}$. It follows that if $d'\equiv1\pmod{3}$, then $$\prod_{\ell\ge5}w_\ell(E_d)=
    \begin{cases}
      \sign(d), & 2\nmid d \text{ or } 4||d, \\
      -\sign(d), & 8||d.
    \end{cases}
$$ If $d'\equiv2\pmod{3}$, then the product of the local root numbers
\begin{equation}
  \label{eq:rootnumberat5}
  \prod_{\ell\ge5}w_\ell(E_d)=
    \begin{cases}
      -\sign(d), & 2\nmid d \text{ or } 4||d, \\
      \sign(d), & 8||d.
    \end{cases}
\end{equation}

  \end{enumerate}
Now the result follows from the product formula $w(E_d)=-w_2(E_d)w_3(E_d)\prod_{\ell\ge5}w_\ell(E_d)$.
\end{proof}

\begin{lemma}\label{lem:rootnumberEd2}
  Assume that:
  \begin{enumerate}
  \item $d$ is a fundamental discriminant.
  \item $d\equiv2\pmod{3}$.
  \end{enumerate}
Then the root number of $E_d$ is given by $$w(E_d)=
\begin{cases}
  \sign(d),& d\equiv2\pmod{9},\\
  -\sign(d),& d\equiv5,8\pmod{9}.
\end{cases}$$
\end{lemma}

\begin{proof}
The proof is similar to Lemma \ref{lem:rootnumberEd} using \cite[\S 9]{Liverance1995}.
  \begin{enumerate}
  \item\label{item:w2}  Since $d$ is a fundamental discriminant, we again have the formula (\ref{eq:rootnumberat2}).
  \item Notice that $-432d=3^3\cdot (-16d)$. Its prime-to-3 part $-16d$ satisfies $-16d\equiv\pm2,1\pmod{9}$ if and only if $d\equiv\pm1,5\pmod{9}$. It follows that the local root number $$w_3(E_d)=
    \begin{cases}
      +1, & d\equiv2\pmod{9},\\
      -1, & d\equiv5,8\pmod{9}.
    \end{cases}
$$
  \item Since $d\equiv2\pmod{3}$, we again have the formula (\ref{eq:rootnumberat5}).
  \end{enumerate}
Now the result again follows from the product formula.
\end{proof}

\subsection{Weak Goldfeld conjecture for $\{E_d\}$}

Since $E_d$ is CM, we know that its conductor $N(E_d)=N_\mathrm{add}(E_d)$. When $d$ is a fundamental discriminant, the curve $E_d$ has additive reduction exactly at the prime factors of $3d$.

% \begin{definition}
% For any non-square integer $D$, we denote by $h_3(D):=|\mathrm{Cl}(\mathbb{Q}(\sqrt{D}))[3]|$ the 3-class number of the quadratic field $\mathbb{Q}(\sqrt{D})$.  
% \end{definition}

\begin{theorem}\label{thm:sextic}
  Let $K=\mathbb{Q}(\sqrt{d_K})$ be an imaginary quadratic field satisfying the Heegner hypothesis with respect to $3d$. Let $P_d\in E_d(K)$ be the associated Heegner point. Assume that:
  \begin{enumerate}
  \item $d$ is a fundamental discriminant.
  \item $d\equiv2\pmod{3}$ or $d\equiv3\pmod{9}$.
  \item\label{item:1} If $d>0$, then $h_3(-3d)=h_3(d_Kd)=1$. If $d<0$, then $h_3(d)=h_3(-3d_Kd)=1$.
  \end{enumerate}
  Then
  \begin{equation}
    \label{eq:logEd}
    \log_{\omega_{E_d}}P_d\not\equiv0\pmod{3}.
  \end{equation}
In particular, $P_d$ is of infinite order and $E_d/K$ has both analytic and algebraic rank one.
\end{theorem}

\begin{proof}
It follows by applying Theorem \ref{thm:Heegnercorollary} for $p=3$ and noticing that $|\tilde E_d^\mathrm{ns}(\mathbb{F}_3)|=3$ since $E_d$ has additive reduction at 3. It remains to check that all the assumptions of Theorem \ref{thm:Heegnercorollary} are satisfied. By Lemma \ref{lem:mod3rep}, we have $E[3]$ is reducible with $\psi=\psi_d$. % Since $d$ is a fundamental discriminant, we know that $f(\psi)=d$, which divides $N_\mathrm{add}(E_d)$.
The condition that $\psi(3)\ne1$ and $(\psi^{-1}\omega)(3)\ne1$ is equivalent to that $d\equiv2\pmod{3}$ or $d\equiv3\pmod{9}$. For $\ell\ne3$ and $\ell|N_\mathrm{add}(E_d)$, we have $\ell|d$, so $\psi_d(\ell)=0$. Finally, the requirement on the trivial 3-class numbers is exactly the assumption that $3\nmid B_{1,\psi_0^{-1}\varepsilon_K}B_{1,\psi_0\omega^{-1}}$ by noticing that $$(\psi_d)_0=
\begin{cases}
\psi_d, & d>0, \\
\psi_{d_Kd}, & d<0,
\end{cases}$$
and using the formula for the Bernoulli numbers (\ref{eq:classnumberformula}) (see also Corollary \ref{cor:classnumberandbernoulli}).
\end{proof}

\begin{corollary}\label{mod9corollary} Assume we are in the situation of Theorem \ref{thm:sextic}.
  \begin{enumerate}
  \item If $d>0$ and $d\equiv2\pmod{9}$, or $d<0$ and $d\equiv3,5,8\pmod{9}$, then $$r_\mathrm{an}(E_d/\mathbb{Q})=0,\quad r_\mathrm{an}(E_d^{(d_K)}/\mathbb{Q})=1.$$
  \item If $d<0$ and $d\equiv2\pmod{9}$, or $d>0$ and $d\equiv3,5,8\pmod{9}$, then $$r_\mathrm{an}(E_d/\mathbb{Q})=1,\quad r_\mathrm{an}(E_d^{(d_K)}/\mathbb{Q})=0.$$
  \end{enumerate}
\end{corollary}

\begin{proof}
  It follows immediately from Theorem \ref{thm:sextic} using the root number calculation in Lemmas \ref{lem:rootnumberEd} and \ref{lem:rootnumberEd2}.
\end{proof}

\begin{corollary}\label{thm:sexticdensity}
  The weak Goldfeld's conjecture holds for the sextic twists family $\{E_d\}$. In fact, $E_d$ has analytic rank $0$ (resp. $1$) for at least 1/6 of fundamental discriminants $d$.
%   \begin{enumerate}
%   \item $E_d$ has analytic rank $0$ for at least 1/12 of fundamental discriminants $d$, and 
%   \item $E_d$ has analytic rank $1$ for at least 1/4 of fundamental discriminants $d$.
%   \end{enumerate}
\end{corollary}

\begin{proof}By Theorem \ref{realtwistpositiveproportion}, at least 1/3 of all (positive or negative) fundamental discriminants $d$ satisfy the assumptions of Theorem \ref{thm:sextic}, and by Remark \ref{atleastoneremark}, for each of these $d$ there is at least one imaginary quadratic field $K$ satisfying the Heegner hypothesis with respect to $3d$ and such that $h_3(d_Kd) = 1$ if $d > 0$ and $h_3(-3d_Kd) = 1$ if $d < 0$. Thus $d$ and $K$ satisfy all of the assumptions of Theorem \ref{thm:sextic}. The final part of Theorem \ref{realtwistpositiveproportion} implies that 1/4 of the fundamental discriminants $d$ considered above (which in turn comprise 1/3 of all fundamental discriminants) satisfy $d \equiv i \pmod{9}$, for each $i \in \{2,3,5,8\}$. Moreover 1/2 of these $d$ give $r_\mathrm{an}(E_d)=0$ (resp. 1) by Corollary \ref{mod9corollary}.  The desired density 1/6 then follows.
\end{proof}

\begin{remark}
One can also obtain $r_\mathrm{an}(E_d)\in \{0,1\}$ for many $d$'s which are not fundamental discriminants. From the proof of Theorem \ref{thm:sextic} one sees that the fundamental discriminant assumption can be relaxed by allowing the exponent of prime factors of $d$ to be 3 or 5 (all we use is that $\mathbb{Q}(\sqrt{d})$ is ramified exactly at the prime factors of $d$). We assume $d$ is a fundamental discriminant only to simplify the root number computation in Lemmas \ref{lem:rootnumberEd} and \ref{lem:rootnumberEd2}. % However, in general the conclusion (\ref{eq:logEd}) may fail when removing any of the three assumptions.
\end{remark}

\subsection{The 3-part of the BSD conjecture over $K$}

The goal of this subsection is to prove the following theorem.  

\begin{theorem}\label{thm:BSD3}
 Assume we are in the situation of Theorem \ref{thm:sextic}. Assume the Manin constant of $E_d$ is coprime to 3. Then BSD(3) is true for $E_d/K$.
\end{theorem}

By the Gross--Zagier formula, the BSD conjecture for $E_d/K$ is equivalent to the equality (\cite[V.2.2]{Gross1986}) 
\begin{equation}\label{eq:BSDHeegner}u_K\cdot c_{E_d}\cdot \prod_{\ell\mid N(E_d)}c_\ell(E_d)\cdot  |\Sha(E_d/K)|^{1/2}=[E_d(K): \mathbb{Z} P_d],
\end{equation}
where $u_K=|\mathcal{O}_K^\times/\{\pm1\}|$, $c_{E_d}$ is the Manin constant of $E_d/\mathbb{Q}$, $c_\ell(E_d)=[E_d(\mathbb{Q}_\ell):E_d^0(\mathbb{Q}_\ell)]$ is the local Tamagawa number of $E_d$ and $[E_d(K): \mathbb{Z}P_d]$ is the index of the Heegner point $P_d\in E_d(K)$.

From now on assume we are in the situation of Theorem \ref{thm:sextic}. Since 3 splits in $K$, we know $K\ne \mathbb{Q}(\sqrt{-1})$ or $\mathbb{Q}(\sqrt{-3})$, so $u_K=1$. Therefore the BSD conjecture for $E_d/K$ is equivalent to the equality
\begin{equation}
  \label{eq:grosszagier}
  \prod_{\ell\mid N(E_d)}c_\ell(E_d)\cdot  |\Sha(E_d/K)|^{1/2}=\frac{[E_d(K): \mathbb{Z} P_d]}{c_{E_d}}.
\end{equation}
We will prove BSD(3) by computing the 3-part of both sides of (\ref{eq:grosszagier}) explicitly.

\begin{lemma}\label{lem:torsion}
 We have  $E_d(K)[3]=0$.
\end{lemma}

\begin{proof}
By Lemma \ref{lem:mod3rep}, we have $E_d[3]^\mathrm{ss}\cong \mathbb{F}_3(\psi_d) \oplus \mathbb{F}_3(\psi_d\omega)$. Since neither $\psi_d$ nor  $\psi_d\omega$ becomes trivial when restricted to $G_K$, we know that $E_d(K)[3]=0$.
\end{proof}

\begin{lemma}\label{lem:tamagawa}
 If $\ell|N(E_d)$ and $\ell\ne3$ (equivalently, $\ell|d$), then $3\nmid c_\ell(E_d)$.
\end{lemma}

\begin{proof}
 By Lemma \ref{lem:mod3rep}, we have $E_d[3]^\mathrm{ss}\cong \mathbb{F}_3(\psi_d) \oplus \mathbb{F}_3(\psi_d\omega)$. Because $\psi_d$ and $\psi_d\omega$ are both nontrivial at $\ell$ (in fact, ramified at $\ell$), we know that $E_d(\mathbb{Q}_\ell)[3]=0$. Since $E_d(\mathbb{Q}_\ell)$ has a pro-$\ell$-subgroup ($\ell\ne3$) of finite index and $E_d(\mathbb{Q}_\ell)$ has trivial 3-torsion, we know that $3\nmid c_\ell(E_d)$.
\end{proof}

\begin{definition}
Let $F$ be any number field. Let $\mathcal{L}=\{\mathcal{L}_v\}$ be a collection of subspaces $L_v\subseteq H^1(F_v, E_d[3])$, where $v$ runs over all places of $L$. We say $\mathcal{L}$ is a collection of \emph{local conditions}  if for almost all $v$, we have $\mathcal{L}_v=H^1_\mathrm{ur}(F_v, E_d[3])$ is the unramified subspace. Notice that $H^1(F_v, E_d[3])=0$, if $v\mid\infty$. We define the \emph{Selmer group cut out by the local conditions $\mathcal{L}$} to be $$H_\mathcal{L}^1(F, E_d[3]):=\{x\in H^1(F, E_d[3]): \mathrm{res}_v(x)\in\mathcal{L}_v,\text{for all } v\}.$$ We will consider the following four types of local conditions:
  \begin{enumerate}
  \item The \emph{Kummer} conditions $\mathcal{L}$ given by $\mathcal{L}_v=\mathrm{im}\left(E(F_v)/3E(F_v)\rightarrow H^1(F_v, E_d[3])\right)$. The 3-Selmer group $\Sel_3(E_d/F)=H_\mathcal{L}^1(F, E_d[3])$ is cut out by the Kummer conditions.
\item The \emph{unramified} conditions $\mathcal{U}$ given by $\mathcal{U}_v=H^1_\mathrm{ur}(F_v, E_d[3])$.
\item The \emph{strict} conditions $\mathcal{S}$ given by $\mathcal{S}_v=\mathcal{U}_v$ for $v\nmid3$ and $\mathcal{S}_v=0$ for $v|3$.
\item The \emph{relaxed} conditions $\mathcal{R}$ given by $\mathcal{R}_v=\mathcal{U}_v$ for $v\nmid3$ and $\mathcal{R}_v=H^1(F_v, E_d[3])$ for $v|3$.
\end{enumerate}
\end{definition}

\begin{lemma}\label{lem:HU1trivial}
  $H^1_\mathcal{U}(K, E_d[3])=H^1_\mathcal{S}(K,E_d[3])=0$.
\end{lemma}

\begin{proof}
  By Shapiro's lemma, we have $$H^1_\mathcal{U}(K,E_d[3])\cong H^1_\mathcal{U}(\mathbb{Q}, E_d[3]) \oplus H^1_\mathcal{U}(\mathbb{Q}, E_d^{(d_K)}[3]).$$ By Lemma \ref{lem:mod3rep}, we have an exact sequence $$\cdots\rightarrow H^1(\mathbb{Q}, \mathbb{F}_3(\psi_d))\rightarrow H^1(\mathbb{Q}, E_d[3])\rightarrow H^1(\mathbb{Q}, \mathbb{F}_3(\psi_d\omega))\rightarrow\cdots.$$ Restricting to the unramfied Selmer group we obtain a map $$H^1_\mathcal{U}(\mathbb{Q}, E_d[3])\rightarrow H^1(\mathbb{Q}, \mathbb{F}_3(\psi_d\omega))$$ whose kernel and image consist of everywhere unramfied classes. It follows from class field theory that $$|H_\mathcal{U}^1(\mathbb{Q},E_d[3])| \le h_3(d)\cdot h_3(-3 d).$$ Similarly, we have $$|H_\mathcal{U}^1(\mathbb{Q},E_d^{(d_K)}[3])|\le h_3(d_Kd)\cdot h_3(-3d_Kd).$$ By the assumptions on the 3-class numbers in Theorem \ref{thm:sextic} and Scholz's reflection theorem (\cite{Scholz1932}, see also \cite[10.2]{Washington1997}), we know that the four 3-class numbers appearing above are all trivial. Hence $H^1_\mathcal{U}(K, E_d[3])=0$. Since by definition we have $$H^1_\mathcal{S}(K, E_d[3])\subseteq H^1_\mathcal{U}(K, E_d[3]),$$ we also know that $H^1_\mathcal{S}(K, E_d[3])=0$.
\end{proof}

\begin{lemma}\label{lem:relaxed}
  $\dim H^1_\mathcal{R}(K,E_d[3])=2$.
\end{lemma}

\begin{proof}
  It follows from \cite[Theorem 2.18]{Darmon1997} that
  \begin{equation}\label{eq:DDT}
\dim H^1_\mathcal{R}(K,E_d[3])-\dim H^1_\mathcal{S}(K,E_d[3])=\frac{1}{2}\sum_{v|3}\dim \mathcal{R}_v.
  \end{equation}
Consider $v|3$.  Since 3 is split in $K$, we know that $H^1(K_v, E_d[3])\cong H^1(\mathbb{Q}_3,E_d[3])$. By Lemma \ref{lem:mod3rep} that $E_d[3]^\mathrm{ss}\cong \mathbb{F}_3(\psi_d)\oplus\mathbb{F}_3(\psi_d\omega)$. Since $\psi_d(3)\ne1$ and $\psi_d\omega(3)\ne1$, we know that $$H^0(\mathbb{Q}_3, E_d[3])=H^2(\mathbb{Q}_3, E_d[3])=0.$$ It follows from the Euler characteristic formula that  $$\dim H^1(\mathbb{Q}_3, E_d[3])=2.$$ Namely, $\dim \mathcal{R}_v=2$. The result then follows from Lemma \ref{lem:HU1trivial} and the formula (\ref{eq:DDT}).
\end{proof}

\begin{lemma}\label{lem:3Selmer}
$\Sel_3(E_d/K)\cong \mathbb{Z}/3 \mathbb{Z}$. In particular, $\Sha(E_d/K)[3]=0$.
\end{lemma}

\begin{proof}
We claim that $\mathcal{L}_v=\mathcal{U}_v$ for any $v\nmid3$. In fact:
  \begin{enumerate}
  \item If $v\nmid 3d\infty$, then $E_d$ has good reduction at $v$ and so $\mathcal{L}_v=H^1_\mathrm{ur}(K_v, E_d[3])$ by \cite[Lemma 6]{Gross2012}.
  \item If $v|\infty$, then $v$ is complex and $H^1(K_v, E_d[3])=0$. So $\mathcal{L}_v=H^1_\mathrm{ur}(K_v, E_d[3])=0$.
  \item If $v| d$, then $v$ is split in $K$ and thus $K_v\cong \mathbb{Q}_\ell$. By Lemma \ref{lem:tamagawa}, $c_\ell(E)$ is coprime to 3. It follows that  $\mathcal{L}_v=H^1_\mathrm{ur}(K_v, E_d[3])$  by \cite[Lemma 6]{Gross2012}.
  \end{enumerate}  
  It follows from the claim that 
  \begin{equation*}
\Sel_3(E_d/K) \subseteq H^1_\mathcal{R}(K,E_d[3]).    
  \end{equation*} So $\dim\Sel_3(E_d/K)\le2$ by Lemma \ref{lem:relaxed}.

By the Heegner hypothesis, the root number of $E_d/K$ is $-1$. Since the 3-parity conjecture is known for elliptic curves with a 3-isogeny (\cite[Theorem 1.8]{Dokchitser2011}), we know that $\dim\Sel_3(E_d/K)$ is odd and thus must be 1. Hence $\Sel_3(E_d/K)\cong \mathbb{Z}/3 \mathbb{Z}$ as desired.
\end{proof}

\begin{lemma}\label{lem:global3divisibility}
We have $$c_3(E_d)=
\begin{cases}
  3, & d\equiv 2\pmod{9},\\
  1, & d\equiv 3,5,8\pmod{9}.
\end{cases}$$
In either case we have $\ord_3(c_3(E_d))=\ord_3\left(\frac{[E_d(K): \mathbb{Z} P_d]}{c_{E_d}}\right)$.
\end{lemma}

\begin{proof}
  The first part follows directly from Tate's algorithm \cite[IV.9]{Silverman1994} (see also the formula in \cite[0.5]{Satge1986}).

  Suppose $\ord_3(c_3(E_d))=0$. We need to show that $\ord_3([E_d(K): \mathbb{Z}P_d])=0$. If not, then since $E_d(K)[3]=0$ (Lemma \ref{lem:torsion}), we know that there exists some $Q\in E_d(K)$ such that $3 Q=nP_d$ for some $n$ coprime to 3. Let $\omega_{\mathcal{E}_d}$ be the N\'{e}ron differential of $E_d$ and let $\log_{E_d}:=\log_{\omega_{\mathcal{E}_d}}$. By the very definition of the Manin constant we have $c_{E_d}\cdot\omega_{E_d}=\omega_{\mathcal{E}_d}$ and $c_{E_d}\cdot\log_{\omega_{E_d}}=\log_{E_d}$. Since $c_{E_d}$ is assumed to be coprime to 3, we have up to a 3-adic unit, $$\frac{|\tilde E_d^\mathrm{ns}(\mathbb{F}_3)|\cdot\log_{\omega_{E_d}}P_d}{3}=\frac{|\tilde E_d^\mathrm{ns} (\mathbb{F}_3)|\cdot\log_{E_d}P_d}{3}=|\tilde E_d^\mathrm{ns} (\mathbb{F}_3)|\cdot\log_{E_d}(Q).$$ On the other hand, $c_3(E_d)\cdot |\tilde E_d^\mathrm{ns}(\mathbb{F}_3)|\cdot Q$ lies in the formal group $\hat E_d(3\mathcal{O}_{K_3})$ and $\ord_3(c_3(E_d))=0$, we know that $$|\tilde E_d^\mathrm{ns} (\mathbb{F}_3)|\cdot\log_{E_d}(Q)\in 3 \mathcal{O}_{K_3},$$  which contradicts the formula (\ref{eq:logEd}).

  Now suppose $\ord_3(c_3(E_d))=1$. The same argument as the previous case shows that we have $\ord_3([E_d(K): \mathbb{Z} P_d])\le1$. It remains to show that $$\ord_3([E_d(K): \mathbb{Z} P_d])\ne0.$$ Assume otherwise, then the image of $P_d$ in $E_d(K)/3E_d(K)$ is \emph{nontrivial}, and hence its image in $\Sel_3(E_d/K)\cong \mathbb{Z}/3 \mathbb{Z}$ is nontrivial. We now analyze its local Kummer image at 3 and derive a contradiction.

  Since $c_3(E_d)=3$ and $\tilde E_d^\mathrm{ns}(\mathbb{F}_3)=\mathbb{Z}/3 \mathbb{Z}$, we know that $E_d(\mathbb{Q}_3)/\hat E_d(3\mathbb{Z}_3)$ is a group of order 9, so 
  \begin{center}
    $E_d(\mathbb{Q}_3)/\hat E_d(3\mathbb{Z}_3)\cong \mathbb{Z}/9 \mathbb{Z}$ or $\mathbb{Z}/3 \mathbb{Z}\times
    \mathbb{Z}/3 \mathbb{Z}$.
  \end{center}
Since $\dim H^1(\mathbb{Q}_3, E_d[3])=2$ and the local Kummer condition is a maximal isotropic subspace of $H^1(\mathbb{Q}_3, E_d[3])$ under the local Tate pairing, we know that $E_d(\mathbb{Q}_3)/3 E_d(\mathbb{Q}_3)=\mathbb{Z}/3 \mathbb{Z}$. So the only possibility is that
\begin{equation}\label{eq:Z9Z}
  E_d(\mathbb{Q}_3)/\hat E_d(3\mathbb{Z}_3)\cong \mathbb{Z}/9 \mathbb{Z}.
\end{equation}
   Now by the formula (\ref{eq:logEd}), we know that $P_d\not\in \hat E_d(3\mathcal{O}_{K_3})$, but $3P_d\in  \hat E_d(3\mathcal{O}_{K_3})$. Using $K_3\cong \mathbb{Q}_3$ and (\ref{eq:Z9Z}), we deduce that $P_d\in 3 E_d(K_3)$. So the local image of $P_d$ in $E_d(K_3)/3 E_d(K_3)$ is \emph{trivial}.

  Therefore $\Sel_3(E_d/K)$ is equal to the strict Selmer group $H^1_\mathcal{S}(K,E_d[3])$, a contradiction to Lemmas \ref{lem:HU1trivial} and \ref{lem:3Selmer}.
\end{proof}

\begin{proof}[Proof of Theorem \ref{thm:BSD3}]
  Theorem \ref{thm:BSD3} follows immediately from the equivalent formula (\ref{eq:grosszagier}) and Lemmas \ref{lem:tamagawa}, \ref{lem:3Selmer} and \ref{lem:global3divisibility}.\end{proof}

\section{Cubic twists families}

\label{sec:cubic-twists-famil}

In this section we consider the elliptic curve $E_d/\mathbb{Q}:y^2=x^3-432d$ of $j$-invariant 0, where $d$ is any 6th-power-free integer. Recall that for a cube-free positive integer $D$, the $D$-th cubic twist $E_d$ is the curve $E_{dD^2}$ (cf. Definition \ref{def:sextictwists}). For $r\ge0$, we define $$C_r(E_d, X)=\{D<X: D>0\text{ cube-free}, r_\mathrm{an}(E_{dD^2})=r\}$$ to be the counting function for the number of cubic twists of $E_d$ of analytic rank $r$. Recall that by Lemma \ref{lem:mod3rep}, $E_d[3]^\mathrm{ss}\cong \mathbb{F}_3(\psi_d) \oplus \mathbb{F}_3(\psi_d\omega)$.

\begin{theorem}\label{thm:cubictwists}
  Assume for any prime $\ell|N(E_d)$, we have $\psi_d(\ell)\ne1$ and $\psi_d\omega(\ell)\ne1$. Assume there exists an imaginary quadratic field $K$ satisfying the Heegner hypothesis for $N(E_d)$ such that
  \begin{enumerate}
  \item 3 is split in $K$.
  \item If $d>0$, then $h_3(-3d)=h_3(d_Kd)=1$. If $d<0$, then $h_3(d)=h_3(-3d_Kd)=1$.
  \end{enumerate}
 Then for $r\in \{0,1\}$, we have $$C_r(E_d, X)\gg \frac{X}{\log^{7/8}(X)}.$$ 
\end{theorem}

\begin{remark}
  Notice that when $3\nmid d$ is a fundamental discriminant, the conditions $\psi_d(\ell)\ne1$ and $\psi_d\omega(\ell)\ne1$ for $\ell|N(E_d)$ are automatically satisfied.
\end{remark}

\begin{proof}
  We consider the following set $\mathcal{S}$ consisting of primes $\ell\nmid 6N(E_d)$ such that
  \begin{enumerate}
  \item $\ell$ is split in $K$.
  \item $\psi_d(\ell)=-1$ ($\ell$ is inert in $\mathbb{Q}(\sqrt{d})$).
  \item $\omega(\ell)=1$ ($\ell$ is split in $\mathbb{Q}(\sqrt{-3})$).
  \end{enumerate}
  Since our assumption implies that the three quadratic fields $K$, $\mathbb{Q}(\sqrt{d})$ and $\mathbb{Q}(\sqrt{-3})$ are linearly disjoint, we know that the set of primes $\mathcal{S}$ has density $\alpha=(\frac{1}{2})^3=\frac{1}{8}$ by Chebotarev's density theorem. 

 Let $\mathcal{N}$ be the set of integers consisting of square-free products of primes in $\mathcal{S}$.  Then for any $D\in \mathcal{N}$. We have $E_{dD^2}[3]^\mathrm{ss}\cong\mathbb{F}_3(\psi_d) \oplus \mathbb{F}_3(\psi_d\omega)$. For any $\ell|N(E_{dD^2})$, we have $\psi_d(\ell)\ne1$ and $\psi_d\omega(\ell)\ne1$ by construction. The imaginary quadratic field $K$ also satisfies the Heegner hypothesis for $N(E_{dD^2})$. Since the relevant 3-class numbers are trivial, we can apply Theorem \ref{thm:Heegnercorollary} ($p=3$) to $E_{dD^2}$ and conclude that $$r_\mathrm{an}(E_{dD^2}/K)=1.$$ The root number $w(E_{dD^2})$ is $+1$ (resp. $-1$) for a positive proportion of $D\in \mathcal{N}$, so we have for $r\in \{0,1\}$, $$C_r(E_d,X)\gg\#\{D\in \mathcal{N}: D< X\}.$$ By a standard application of Ikehara's tauberian theorem (see \cite[3.3]{KrizLi2016a}), we know that $$\#\{D\in N: D<X\}\sim c\cdot \frac{X}{\log^{1-\alpha}X},$$ for some $c>0$. Here $\alpha=\frac{1}{8}$ is the density of the set of primes $\mathcal{S}$. The results then follow.
\end{proof}

\begin{example}\label{exa:cubictwist108}
  Consider $d=2^2\cdot 3^3=108$. Then $E_d=144a1: y^2=x^3-1$. The field $K=\mathbb{Q}(\sqrt{-23})$ satisfies the Heegner hypothesis for $N=144$ and 3 is split in $K$. We compute the 3-class numbers $h_3(-3 d)=h_3(-1)=1$ and $h_3(d_K d)=h_3(-69)=1$. So the assumptions of Theorem \ref{thm:cubictwists} are satisfied. The set $\mathcal{N}$ in the proof of Theorem \ref{thm:cubictwists} consists of square-free products of the primes $$31, 127, 139, 151, 163, 211, 223, 271, 307, 331, 439, 463, 487, 499,\cdots$$ Notice that $D\in\mathcal{N}$ implies that $D\equiv 1\pmod{3}$. One can then compute the root number of the cubic twist $$E_{dD^2}: y^2=x^3-D^2$$ to be $$w(E_{dD^2})=
  \begin{cases}
    +1, & D\equiv 1,4\pmod{9},\\
    -1, & D\equiv 7\pmod{9}.
  \end{cases}$$
We conclude that for $D\in \mathcal{N}$, $$r_\mathrm{an}(E_{dD^2})=
\begin{cases}
  0, & D\equiv 1,4\pmod{9},\\
  1, & D\equiv 7\pmod{9}.
\end{cases}$$
\end{example}

\bibliographystyle{alpha}
\bibliography{Congruence}

\begin{thebibliography}{{Yoo}15}

\bibitem[Bal14]{Balsam2014}
Nava Balsam.
\newblock The parity of analytic ranks among quadratic twists of elliptic
  curves over number fields, 2014.

\bibitem[BDP13]{Bertolini2013}
Massimo Bertolini, Henri Darmon, and Kartik Prasanna.
\newblock Generalized {H}eegner cycles and {$p$}-adic {R}ankin {$L$}-series.
\newblock {\em Duke Math. J.}, 162(6):1033--1148, 2013.
\newblock With an appendix by Brian Conrad.

\bibitem[BES16]{Bhargava2016}
M.~{Bhargava}, N.~{Elkies}, and A.~{Shnidman}.
\newblock {The average size of the 3-isogeny Selmer groups of elliptic curves
  $y^2 = x^3 + k$}.
\newblock {\em ArXiv e-prints}, October 2016.

\bibitem[BJK09]{Byeon2009}
Dongho Byeon, Daeyeol Jeon, and Chang~Heon Kim.
\newblock Rank-one quadratic twists of an infinite family of elliptic curves.
\newblock {\em J. Reine Angew. Math.}, 633:67--76, 2009.

\bibitem[BKLS17]{Bhargava2017}
M.~{Bhargava}, Z.~{Klagsbrun}, R.~J. {Lemke Oliver}, and A.~{Shnidman}.
\newblock {Three-isogeny Selmer groups and ranks of abelian varieties in
  quadratic twist families over a number field}.
\newblock {\em ArXiv e-prints}, September 2017.

\bibitem[{Bro}17]{Browning2017}
T.~D. {Browning}.
\newblock {Many cubic surfaces contain rational points}.
\newblock {\em ArXiv e-prints}, January 2017.

\bibitem[BSZ14]{Bhargava2014}
M.~{Bhargava}, C.~{Skinner}, and W.~{Zhang}.
\newblock {A majority of elliptic curves over $\mathbb Q$ satisfy the Birch and
  Swinnerton-Dyer conjecture}.
\newblock {\em ArXiv e-prints}, July 2014.

\bibitem[DD11]{Dokchitser2011}
Tim Dokchitser and Vladimir Dokchitser.
\newblock Root numbers and parity of ranks of elliptic curves.
\newblock {\em J. Reine Angew. Math.}, 658:39--64, 2011.

\bibitem[DDT97]{Darmon1997}
Henri Darmon, Fred Diamond, and Richard Taylor.
\newblock Fermat's last theorem.
\newblock In {\em Elliptic curves, modular forms \& {F}ermat's last theorem
  ({H}ong {K}ong, 1993)}, pages 2--140. Int. Press, Cambridge, MA, 1997.

\bibitem[DH71]{Davenport1971}
H.~Davenport and H.~Heilbronn.
\newblock On the density of discriminants of cubic fields. {II}.
\newblock {\em Proc. Roy. Soc. London Ser. A}, 322(1551):405--420, 1971.

\bibitem[Fou93]{Fouvry1993}
{\'E}.~Fouvry.
\newblock Sur le comportement en moyenne du rang des courbes {$y^2=x^3+k$}.
\newblock In {\em S\'eminaire de {T}h\'eorie des {N}ombres, {P}aris, 1990--91},
  volume 108 of {\em Progr. Math.}, pages 61--84. Birkh\"auser Boston, Boston,
  MA, 1993.

\bibitem[Gol79]{Goldfeld1979}
Dorian Goldfeld.
\newblock Conjectures on elliptic curves over quadratic fields.
\newblock In {\em Number theory, {C}arbondale 1979 ({P}roc. {S}outhern
  {I}llinois {C}onf., {S}outhern {I}llinois {U}niv., {C}arbondale, {I}ll.,
  1979)}, volume 751 of {\em Lecture Notes in Math.}, pages 108--118. Springer,
  Berlin, 1979.

\bibitem[GP12]{Gross2012}
Benedict~H. Gross and James~A. Parson.
\newblock On the local divisibility of {H}eegner points.
\newblock In {\em Number theory, analysis and geometry}, pages 215--241.
  Springer, New York, 2012.

\bibitem[Gro80]{Gross1980}
Benedict~H. Gross.
\newblock On the factorization of {$p$}-adic {$L$}-series.
\newblock {\em Invent. Math.}, (1):83--95, 1980.

\bibitem[Gro84]{Gross1984}
Benedict~H. Gross.
\newblock Heegner points on {$X_0(N)$}.
\newblock In {\em Modular forms ({D}urham, 1983)}, Ellis Horwood Ser. Math.
  Appl.: Statist. Oper. Res., pages 87--105. Horwood, Chichester, 1984.

\bibitem[Gro11]{Gross2011}
Benedict~H. Gross.
\newblock Lectures on the conjecture of {B}irch and {S}winnerton-{D}yer.
\newblock In {\em Arithmetic of {$L$}-functions}, volume~18 of {\em IAS/Park
  City Math. Ser.}, pages 169--209. Amer. Math. Soc., Providence, RI, 2011.

\bibitem[GZ86]{Gross1986}
Benedict~H. Gross and Don~B. Zagier.
\newblock Heegner points and derivatives of {$L$}-series.
\newblock {\em Invent. Math.}, 84(2):225--320, 1986.

\bibitem[HB94]{Heath-Brown1994}
D.~R. Heath-Brown.
\newblock The size of {S}elmer groups for the congruent number problem. {II}.
\newblock {\em Invent. Math.}, 118(2):331--370, 1994.
\newblock With an appendix by P. Monsky.

\bibitem[HB04]{Heath-Brown2004}
D.~R. Heath-Brown.
\newblock The average analytic rank of elliptic curves.
\newblock {\em Duke Math. J.}, 122(3):591--623, 2004.

\bibitem[HT93]{HidaTilouine1993}
H.~Hida and J.~Tilouine.
\newblock Anti-cyclotomic {K}atz {$p$}-adic {$L$}-functions and congruence
  modules.
\newblock {\em Ann. Sci. \'Ecole Norm. Sup. (4)}, 26(2):189--259, 1993.

\bibitem[Jam98]{James1998}
Kevin James.
\newblock {$L$}-series with nonzero central critical value.
\newblock {\em J. Amer. Math. Soc.}, 11(3):635--641, 1998.

\bibitem[Jam99]{James1999}
Kevin James.
\newblock Elliptic curves satisfying the {B}irch and {S}winnerton-{D}yer
  conjecture mod {$3$}.
\newblock {\em J. Number Theory}, 76(1):16--21, 1999.

\bibitem[Kan13]{Kane2013}
Daniel Kane.
\newblock On the ranks of the 2-{S}elmer groups of twists of a given elliptic
  curve.
\newblock {\em Algebra Number Theory}, 7(5):1253--1279, 2013.

\bibitem[Kat76]{Katz1976}
Nicholas~M. Katz.
\newblock $p$-adic {I}nterpolation of {R}eal {A}nalytic {E}isenstein {S}eries.
\newblock {\em Ann. of Math.}, 104(3):459--571, 1976.

\bibitem[Kat78]{Katz1978}
Nicholas~M. Katz.
\newblock {$p$}-adic {$L$}-functions for {CM} fields.
\newblock {\em Invent. Math.}, 49(3):199--297, 1978.

\bibitem[KL16]{KrizLi2016a}
D.~{Kriz} and C.~{Li}.
\newblock {Congruences between Heegner points and quadratic twists of elliptic
  curves}.
\newblock {\em ArXiv e-prints}, June 2016.

\bibitem[Kob13]{Kobayashi2013}
Shinichi Kobayashi.
\newblock The {$p$}-adic {G}ross-{Z}agier formula for elliptic curves at
  supersingular primes.
\newblock {\em Invent. Math.}, 191(3):527--629, 2013.

\bibitem[Kri16]{Kriz2016}
Daniel Kriz.
\newblock Generalized {H}eegner cycles at {E}isenstein primes and the {K}atz
  p-adic {L}-function.
\newblock {\em Algebra Number Theory}, 10(2):309--374, 2016.

\bibitem[KS99]{Katz1999}
Nicholas~M. Katz and Peter Sarnak.
\newblock {\em Random matrices, {F}robenius eigenvalues, and monodromy},
  volume~45 of {\em American Mathematical Society Colloquium Publications}.
\newblock American Mathematical Society, Providence, RI, 1999.

\bibitem[Liv95]{Liverance1995}
Eric Liverance.
\newblock A formula for the root number of a family of elliptic curves.
\newblock {\em J. Number Theory}, 51(2):288--305, 1995.

\bibitem[LLT16]{Li2016}
Y.~{Li}, Y.~{Liu}, and Y.~{Tian}.
\newblock {On The Birch and Swinnerton-Dyer Conjecture for CM Elliptic Curves
  over $\mathbb{Q}$}.
\newblock {\em ArXiv e-prints}, May 2016.

\bibitem[LZZ15]{Liu2014}
Y.~{Liu}, S.~{Zhang}, and W.~{Zhang}.
\newblock {On $p$-adic Waldspurger formula}.
\newblock {\em ArXiv e-prints}, November 2015.

\bibitem[Maz79]{Mazur1979}
B.~Mazur.
\newblock On the arithmetic of special values of {$L$} functions.
\newblock {\em Invent. Math.}, 55(3):207--240, 1979.

\bibitem[Nek90]{Nekovavr1990}
Jan Nekov{\'a}{\v{r}}.
\newblock Class numbers of quadratic fields and {S}himura's correspondence.
\newblock {\em Math. Ann.}, 287(4):577--594, 1990.

\bibitem[NH88]{Nakagawa1988}
Jin Nakagawa and Kuniaki Horie.
\newblock Elliptic curves with no rational points.
\newblock {\em Proc. Amer. Math. Soc.}, 104(1):20--24, 1988.

\bibitem[Ono98]{Ono1998a}
Ken Ono.
\newblock A note on a question of {J}. {N}ekov\'a\v r and the {B}irch and
  {S}winnerton-{D}yer conjecture.
\newblock {\em Proc. Amer. Math. Soc.}, 126(10):2849--2853, 1998.

\bibitem[OS98]{Ono1998}
Ken Ono and Christopher Skinner.
\newblock Non-vanishing of quadratic twists of modular {$L$}-functions.
\newblock {\em Invent. Math.}, 134(3):651--660, 1998.

\bibitem[PR87]{Perrin-Riou1987}
Bernadette Perrin-Riou.
\newblock Points de {H}eegner et d\'eriv\'ees de fonctions {$L$} {$p$}-adiques.
\newblock {\em Invent. Math.}, 89(3):455--510, 1987.

\bibitem[PR04]{Pollack2004}
Robert Pollack and Karl Rubin.
\newblock The main conjecture for {CM} elliptic curves at supersingular primes.
\newblock {\em Ann. of Math. (2)}, 159(1):447--464, 2004.

\bibitem[Rub83]{Rubin1983}
Karl Rubin.
\newblock Congruences for special values of {$L$}-functions of elliptic curves
  with complex multiplication.
\newblock {\em Invent. Math.}, 71(2):339--364, 1983.

\bibitem[Rub91]{Rubin1991}
Karl Rubin.
\newblock The ``main conjectures'' of {I}wasawa theory for imaginary quadratic
  fields.
\newblock {\em Invent. Math.}, 103(1):25--68, 1991.

\bibitem[{Sag}16]{sage}
The {Sage Developers}.
\newblock {\em {S}ageMath, the {S}age {M}athematics {S}oftware {S}ystem
  ({V}ersion 7.2)}, 2016.
\newblock {\tt http://www.sagemath.org}.

\bibitem[Sat86]{Satge1986}
Philippe Satg{\'e}.
\newblock Groupes de {S}elmer et corps cubiques.
\newblock {\em J. Number Theory}, 23(3):294--317, 1986.

\bibitem[Sch32]{Scholz1932}
Arnold Scholz.
\newblock \"{U}ber die {B}eziehung der {K}lassenzahlen quadratischer {K}\"orper
  zueinander.
\newblock {\em J. Reine Angew. Math.}, 166:201--203, 1932.

\bibitem[{Shn}17]{Shnidman2017}
A.~{Shnidman}.
\newblock {Quadratic twists of abelian varieties with real multiplication}.
\newblock {\em ArXiv e-prints}, October 2017.

\bibitem[Sil94]{Silverman1994}
Joseph~H. Silverman.
\newblock {\em Advanced topics in the arithmetic of elliptic curves}, volume
  151 of {\em Graduate Texts in Mathematics}.
\newblock Springer-Verlag, New York, 1994.

\bibitem[{Smi}16]{Smith2016}
A.~{Smith}.
\newblock {The congruent numbers have positive natural density}.
\newblock {\em ArXiv e-prints}, March 2016.

\bibitem[{Smi}17]{Smith2017}
A.~{Smith}.
\newblock {$2^\infty$-Selmer groups, $2^\infty$-class groups, and Goldfeld's
  conjecture}.
\newblock {\em ArXiv e-prints}, February 2017.

\bibitem[Tay00]{Taya2000}
Hisao Taya.
\newblock Iwasawa invariants and class numbers of quadratic fields for the
  prime {$3$}.
\newblock {\em Proc. Amer. Math. Soc.}, 128(5):1285--1292, 2000.

\bibitem[TYZ14]{Tian2014a}
Y.~{Tian}, X.~{Yuan}, and S.~{Zhang}.
\newblock {Genus Periods, Genus Points and Congruent Number Problem}.
\newblock {\em ArXiv e-prints}, November 2014.

\bibitem[Vat98]{Vatsal1998}
V.~Vatsal.
\newblock Rank-one twists of a certain elliptic curve.
\newblock {\em Math. Ann.}, 311(4):791--794, 1998.

\bibitem[Vat99]{Vatsal1999}
V.~Vatsal.
\newblock Canonical periods and congruence formulae.
\newblock {\em Duke Math. J.}, 98(2):397--419, 1999.

\bibitem[Was97]{Washington1997}
Lawrence~C. Washington.
\newblock {\em Introduction to cyclotomic fields}, volume~83 of {\em Graduate
  Texts in Mathematics}.
\newblock Springer-Verlag, New York, second edition, 1997.

\bibitem[Wat07]{Watkins2007}
Mark Watkins.
\newblock Rank distribution in a family of cubic twists.
\newblock In {\em Ranks of elliptic curves and random matrix theory}, volume
  341 of {\em London Math. Soc. Lecture Note Ser.}, pages 237--246. Cambridge
  Univ. Press, Cambridge, 2007.

\bibitem[{Yoo}15]{Yoo2015}
Hwajong {Yoo}.
\newblock Non-optimal levels of a reducible mod l modular representation, 2015.

\bibitem[ZK87]{Zagier1987}
D.~Zagier and G.~Kramarz.
\newblock Numerical investigations related to the {$L$}-series of certain
  elliptic curves.
\newblock {\em J. Indian Math. Soc. (N.S.)}, 52:51--69 (1988), 1987.

\end{thebibliography}

\end{document}